\documentclass{amsart}
\usepackage[utf8]{inputenc}
\usepackage[T1]{fontenc}
\RequirePackage{amsmath}
\RequirePackage{amssymb}
\RequirePackage{amsxtra}
\RequirePackage{amsfonts}
\RequirePackage{latexsym}
\RequirePackage{euscript}
\RequirePackage{amscd}
\RequirePackage{amsthm}
\RequirePackage{xypic}
\usepackage{stmaryrd}
\usepackage{mathtools}
\xyoption{all}

\newcommand{\depth}{\operatorname{depth}}
\newcommand{\Rloc}{R^{\operatorname{loc}}}
\newcommand{\ob}{\operatorname{ob}}
\newcommand{\Sets}{\underline{Sets}}
\def\Kbar{\overline{K}}
\def\Ebar{\overline{E}}
\def\Lbar{\overline{L}}
 \DeclareMathOperator{\nInd}{n-Ind}

\usepackage{tikz}
\usetikzlibrary{matrix,arrows,decorations.pathmorphing}
\usepackage{tikz-cd}
\makeatletter
\newlength\xvec@height%
\newlength\xvec@depth%
\newlength\xvec@width%
\newcommand{\xvec}[2][]{%
  \ifmmode%
    \settoheight{\xvec@height}{$#2$}%
    \settodepth{\xvec@depth}{$#2$}%
    \settowidth{\xvec@width}{$#2$}%
  \else%
    \settoheight{\xvec@height}{#2}%
    \settodepth{\xvec@depth}{#2}%
    \settowidth{\xvec@width}{#2}%
  \fi%
  \def\xvec@arg{#1}%
  \def\xvec@dd{:}%
  \def\xvec@d{.}%
  \raisebox{.2ex}{\raisebox{\xvec@height}{\rlap{%
    \kern.05em
    \begin{tikzpicture}[scale=1]
    \pgfsetroundcap
    \draw (.05em,0)--(\xvec@width-.05em,0);
    \draw (\xvec@width-.05em,0)--(\xvec@width-.15em, .075em);
    \draw (\xvec@width-.05em,0)--(\xvec@width-.15em,-.075em);
    \ifx\xvec@arg\xvec@d%
      \fill(\xvec@width*.45,.5ex) circle (.5pt);%
    \else\ifx\xvec@arg\xvec@dd%
      \fill(\xvec@width*.30,.5ex) circle (.5pt);%
      \fill(\xvec@width*.65,.5ex) circle (.5pt);%
    \fi\fi%
    \end{tikzpicture}%
  }}}%
  #2%
}
\makeatother

\usepackage{dsfont}

\let\emptyset\varnothing

\def\A{\mathbb A}

\def\C{\mathbb C}
\def\F{\mathbb F}

\def\Q{\mathbb{Q}}
\def\R{\mathbb{R}}
\def\T{\mathbb{T}}

\def\Z{\mathbb{Z}}

\def\Qbar{\overline{\Q}}

\def\Zhat{\widehat{\Z}}

\def\m{\mathfrak m}
\def\n{\mathfrak n}

\def\chibar{\overline{\chi}}

\def\mult{\mathrm{m}}

\def\id{\mathrm{id}}

\def\et{\mathrm{\acute{e}t}}

\def\red{\mathrm{red}}

\def\ab{\mathrm{ab}}
\newcommand{\Kab}{K^\ab}
\def\nr{\mathrm{ur}}

\def\SL{\mathrm{SL}}

\def\GL{\operatorname{GL}}

\def\Gal{\mathrm{Gal}}

\def\Aut{\mathrm{Aut}}
\def\Sym{\mathrm{Sym}}

\def\End{\mathrm{End}}

\def\Art{\mathop{\mathrm{Art}}\nolimits}

\def\Hom{\mathop{\mathrm{Hom}}\nolimits}

\def\Spec{\mathop{\mathrm{Spec}}\nolimits}

\def\Frob{\mathop{\mathrm{Frob}}\nolimits}
\def\Supp{\mathop{\mathrm{Supp}}\nolimits}

\def\rhobar{\overline{\rho}}
\def\rhomod{\rho^{\mathrm{mod}}}

\def\cotimes{\widehat{\otimes}}

\def\WD{\mathrm{WD}}

\def\m{\mathfrak{m}}

\swapnumbers

\newcommand{\onto}{\twoheadrightarrow}

\newcommand{\lra}{\longrightarrow}

\newcommand{\into}{\hookrightarrow}

\newcommand{\To}{\longrightarrow}
\newcommand{\isoto}{\stackrel{\sim}{\To}}

\newlength{\ownl}

\newcommand{\ad}{{\operatorname{ad}\,}}

\newcommand{\BC}{{\operatorname{BC}\,}}

\newcommand{\chara}{{\operatorname{char}\,}}

\renewcommand{\Im}{{\operatorname{Im}\,}}
\newcommand{\im}{{\operatorname{Im}\,}}

\newcommand{\JL}{{\operatorname{JL}\,}}

\newcommand{\rec}{{\operatorname{rec}}}

\newcommand{\Res}{{\operatorname{Res}}}

\newcommand{\tr}{{\operatorname{tr}\,}}

\newcommand{\val}{{\operatorname{val}\,}}

\newcommand{\PGL}{\operatorname{PGL}}
\newcommand{\PSL}{\operatorname{PSL}}
\newcommand{\Sp}{\operatorname{Sp}}

\newcommand{\Fsemis}{{\operatorname{F-ss}}}

\newcommand{\cris}{{\operatorname{cr}}}

\newcommand{\loc}{{\operatorname{loc}}}
\newcommand{\semis}{{\operatorname{ss}}}

\newcommand{\univ}{{\operatorname{univ}}}

\newcommand{\cC}{\mathcal{C}}
\renewcommand{\cD}{\mathcal{D}}

\renewcommand{\cH}{\mathcal{H}}

\newcommand{\cJ}{\mathcal{J}}

\newcommand{\cO}{\mathcal{O}}

\newcommand{\cP}{\mathcal{P}}
\newcommand{\cp}{\mathcal{P}}

\newcommand{\cS}{\mathcal{S}}

\newcommand{\gp}{{\mathfrak{p}}}

\newcommand{\barL}{\overline{{L}}}
\newcommand{\barM}{\overline{{M}}}

\newcommand{\tE}{\widetilde{{E}}}

\newcommand{\tL}{\widetilde{{L}}}

\newcommand{\alphabar   }{\overline{\alpha  }}   
\newcommand{\betabar         }{\overline{\beta}}

 \newcommand{\sigmat   }{\widetilde{\sigma}}

\def\RCS$#1: #2 ${\expandafter\def\csname RCS#1\endcsname{#2}}
\RCS$Revision: 85 $
\RCS$Date: 2011-12-16 18:54:17 -0600 (Fri, 16 Dec 2011) $
 
\newcommand{\HT}{\operatorname{HT}}

 \newcommand{\Qp}{{\Q_p}}

\newcommand{\Zp}{{\Z_p}}
\newcommand{\Zptimes}{{\Z_p^\times}}
\newcommand{\Ql}{{\Q_l}} 
\newcommand{\Qpbar}{{\overline{\Q}_p}}
\newcommand{\Qpbartimes}{{\overline{\Q}_p^\times}}

\newcommand{\Fpbar}{{\overline{\F}_p}}

\newcommand{\Fp}{{\F_p}}

\usepackage{hyperref} 

\newtheorem{thm}[subsection]{Theorem}

\newtheorem{lem}[subsection]{Lemma}

\newtheorem{cor}[subsection]{Corollary}
\newtheorem{conj}[subsection]{Conjecture}

\newtheorem{prop}[subsection]{Proposition}

\theoremstyle{definition}

\newtheorem{defn}[subsection]{Definition}

\theoremstyle{remark}
\newtheorem{remark}[subsection]{Remark}
\newtheorem{rem}[subsection]{Remark}

 \newtheorem{exercise}[subsection]{Exercise}
 \newtheorem{fact}[subsection]{Fact}
\def\numequation{\addtocounter{subsection}{1}\begin{equation}}
\def\nummultline{\addtocounter{subsubsection}{1}\begin{multline}}
\def\anumequation{\addtocounter{subsection}{1}\begin{equation}}

\begin{document}








 \newcommand{\Wedge}{\wedge}
 \newcommand{\Syl}{{\operatorname{Syl}}}

 \newcommand{\GQl}{{G_\Ql}}
 \newcommand{\gql}{{\GQl}}
 \newcommand{\gqq}{{G_{\Q_q}}}
 \newcommand{\iqq}{{I_{\Q_q}}}
 \newcommand{\iql}{{I_\Ql}}

 \newcommand{\ca}{\mathfrak{a}}
 \newcommand{\cb}{\mathfrak{b}}
 \newcommand{\cc}{\mathfrak{c}}
 \newcommand{\cd}{\mathfrak{d}}
 \newcommand{\Zl}{{\Z_l}}
 \newcommand{\tame}{^{\mathrm{tame}}}

 \newcommand{\SO}{\operatorname{SO}}
 \newcommand{\Qq}{{\Q_q}}

\setcounter{tocdepth}{1} 







\title[Modularity lifting theorems]{Modularity lifting theorems}

\author{Toby Gee} \email{toby.gee@imperial.ac.uk} \address{Department of
  Mathematics, Imperial College London}

\date{\today}
\maketitle

\tableofcontents
\section{Introduction}The main aim of these notes is to explain
modularity/automorphy lifting theorems for two-dimensional $p$-adic
representations, using wherever possible arguments that go over to the
(essentially conjugate self-dual) $n$-dimensional case. In particular,
we use improvements on the original Taylor--Wiles method due to
Diamond, Fujiwara and Kisin, and we explain (in the case $n=2$)
Taylor's arguments in~\cite{tay} that avoid the use of Ihara's
lemma. For the most part I ignore the issues which are local at $p$,
focusing on representations which satisfy the Fontaine--Laffaille
condition. These notes are based in part on a lecture course given by
Richard Taylor at Harvard University in 2009, and I am extremely
grateful to him both for the many things I learned from his
course, and for giving me permission to present them here.


These notes were originally written for the 2013 Arizona Winter
School, and I would like to thank the organisers (Bryden Cais, Mirela
Çiperiani, Matthew Papanikolas, Rachel Pries, David Savitt, and Romyar
Sharifi) for the invitation to speak, and the participants for their
feedback. I would like to thank Ana Caraiani for encouraging me
to belatedly publish these notes. In revising for publication I removed some
material relating to the Arizona Winter School project, and for the
most part removed the outdated section on more recent developments
(the reader disappointed by this change should read Calegari's excellent
survey~\cite{calegari2021reciprocity}). I have otherwise maintained
the feel of the original notes. In particular, there are exercises for
the reader (of varying difficulty) throughout the text.

I would also like to thank Kevin Buzzard, Frank Calegari, Ana Caraiani, Kęstutis
Česnavičius, Matthew Emerton, Jessica Fintzen, Jeff Hatley, Florian Herzig, Christian
Johansson, Keenan Kidwell, Alexander Kutzim, Jack Lamplugh, Heejong Lee, Tom Lovering,
Judith Ludwig, James Newton, Sug Woo Shin and Jack Shotton for their
helpful comments on earlier versions of these notes.

Finally, I am extremely grateful to the anonymous referee for their
many helpful corrections and suggestions.
\subsection{Notation}Much of this notation will also be introduced in
the text, but I have tried to collect together various definitions
here, for ease of reading. Throughout these notes, $p>2$ is a prime
greater than two. In the earlier stages of the notes, we discuss
$n$-dimensional $p$-adic and mod $p$ representations, before
specialising to the case $n=2$. When we do so, we assume that $p\nmid
n$. (Of course, in the case $n=2$, this follows from our assumption
that $p>2$.)

If $M$ is a field, we let $G_M$ denote its absolute Galois group
$\Gal(\overline{M}/M)$, where $\overline{M}$ is some choice of
separable closure of~$M$. We
write $\varepsilon_p$ (or just~$\varepsilon$) for the $p$-adic cyclotomic character. We fix an
algebraic closure $\Qbar$ of $\Q$, and regard all algebraic extensions
of $\Q$ as subfields of $\Qbar$. For each prime $p$ we fix an
algebraic closure $\Qpbar$ of $\Qp$, and we fix an embedding
$\Qbar\into\Qpbar$. In this way, if $v$ is a finite place of a number
field $F$, we have a homomorphism $G_{F_v}\into G_F$. We also fix an
embedding $\Qbar\into\C$.  If $L/\Qp$ is algebraic, then we write
$\cO_L$ for the ring of integers of~$L$, and~ $k(L)$ for its residue field.

We normalise the definition of Hodge--Tate weights so that all the
Hodge--Tate weights of the $p$-adic cyclotomic character $\varepsilon_p$
are $-1$.

If~$R$ is a local ring, we write $\m_R$ for the maximal ideal of~$R$.

We let $\zeta_p$ be a primitive $p$th root of unity.

We use the terms ``modularity lifting theorem'' and ``automorphy
lifting theorem'' more or less interchangeably. 
\section{Galois representations}Modularity lifting theorems prove that
certain Galois representations are modular, in the sense that they
come from modular forms. We begin in this first chapter by introducing Galois
representations, and explaining some of their basic properties.
\subsection{Basics of Galois representations (and structure of Galois groups)}
Let $K'/K$ be a (not necessarily finite)
normal and separable extension of fields. Then the Galois group
$\Gal(K'/K)$ is the
group \[\{\sigma\in\Aut(K'):\sigma|_K=\id_K\}.\]This has a natural
topology, making it a compact Hausdorff totally disconnected
topological group; equivalently, it is a profinite group. This can be
expressed by the topological
isomorphism \[\Gal(K'/K)\cong\varprojlim_{\stackrel{K''/K\textrm{ finite
    normal}}{K''\subseteq K'}}\Gal(K''/K),\] where the finite groups $\Gal(K''/K)$ have
the discrete topology.
Then Galois theory gives a bijective correspondence between
intermediate fields $K'\supset K''\supset K$ and closed subgroups
$H\subset\Gal(K'/K)$, with $K''$ corresponding to $\Gal(K'/K'')$ and
$H$ corresponding to $K^H$. (See e.g.\ Section 1.6 of \cite{MR0225922}.)

Fix a separable closure $\Kbar$ of $K$, and write
$G_K:=\Gal(\Kbar/K)$. Let $L$ be a topological field; then a
\emph{Galois representation} is a continuous homomorphism
$\rho:G_K\to\GL_n(L)$ for some $n$. The nature of these representations
depends on the topology on $L$. For example, if $L$ has the discrete
topology, then the image of $\rho$ is finite, and $\rho$ factors
through a finite Galois group $\Gal(K''/K)$.
\begin{exercise}
  If $L=\C$ with the usual topology, then $\rho(G_K)$ is finite, and
  $\rho$ is conjugate to a representation valued in $\GL_n(\Qbar)$.
\end{exercise}
On the other hand, if $L/\Q_p$ is a finite extension with the $p$-adic
topology, then there can be examples with infinite image. The rest of
these notes will be concerned with these $p$-adic representations. For example,
if $p\ne\chara K$, we have the $p$-adic cyclotomic character
$\varepsilon_p:G_K\to\Zptimes$, which is uniquely determined by the
requirement that if $\sigma\in G_K$ and $\zeta\in\Kbar$ with
$\zeta^{p^m}=1$ for some $n$, then
$\sigma(\zeta)=\zeta^{\varepsilon_p(\sigma)\pmod{p^m}}$. More
interesting examples arise from geometry, as we explain in
Section~\ref{subsec: p adic HT} below. 
\begin{fact}
  If $L/\Qp$ is an algebraic extension, and $\rho:G_K\to\GL_n(L)$ is a
  continuous representation, then $\rho(G_K)\subseteq\GL_n(M)$ for
  some $L\supset M\supset\Qp$ with $M/\Qp$ finite.
\end{fact}
\begin{proof}
  This follows from the Baire category theorem; see e.g.\ the proof of
  Corollary 5 of~\cite{MR1845182} for the details.
\end{proof}
\begin{exercise}\label{ex: existence of galois stable lattice}
  If $L/\Qp$ is an algebraic extension, and $\rho:G_K\to\GL_n(L)$ is a
  continuous representation, then $\rho$ is conjugate to a
  representation in $\GL_n(\cO_L)$.
\end{exercise}
Any finite-dimensional Galois representation has a Jordan--H\"older
sequence, and thus a well-defined semisimplification.
\begin{fact}
  Two Galois representations $\rho,\rho':G_K\to\GL_n(L)$ have
  isomorphic semisimplifications if and only if $\rho(g),\rho'(g)$
  have the same characteristic polynomials for each $g\in G_K$. If
  $\chara L=0$ (or indeed if $\chara L>n$), then this is equivalent to $\tr\rho(g)=\tr\rho'(g)$
  for all $g\in G_K$.
\end{fact}
\begin{proof}
  This is a consequence of the Brauer--Nesbitt theorem, \cite[30.16]{MR0144979}
\end{proof}
As a corollary of the previous exercise and fact, we see that $p$-adic
representations have well-defined semi-simplified reductions modulo
$p$. Indeed, given $\rho:G_K\to\GL_n(L)$ with $L/\Qp$ algebraic, we
may conjugate $\rho$ to be valued in $\GL_n(\cO_L)$, reduce modulo the
maximal ideal and semisimplify to get a semisimple representation
$\rhobar:G_K\to\GL_n(k(L))$, whose characteristic
polynomials are determined by those of $\rho$.

\begin{rem}
  We really do have to semisimplify here; to see why, think about the
  reductions modulo $p$ of the matrices $
  \begin{pmatrix}
    1&1\\0&1
  \end{pmatrix}$ and $
  \begin{pmatrix}
    1&p\\0&1
  \end{pmatrix}$.
\end{rem}

\subsection{Local representations with $p\ne l$: the monodromy
  theorem}\label{subsec: monodromy p not l}In this section we will let $K/\Ql$ be a finite extension,
for some prime $l\ne p$. In order to study the representations of
$G_K$, we firstly recall something of the structure of $G_K$ itself;
see e.g.~\cite{MR554237} for further details. Let $\varpi_K$ be a uniformiser
of $\cO_K$,  let $k=k(K)$ denote the residue field of $K$,
and let
$\val_K:K^\times\onto\Z$ be the $\varpi_K$-adic valuation. Let
$|\cdot|_K:=(\# k)^{-\val_K(\cdot)}$ be the corresponding norm. The
action of $G_K$ on $K$ preserves $\val_K$, and thus induces an action
on $k$, so that we have a homomorphism $G_K\to G_k$, and in fact a
short exact sequence \[0\to I_K\to G_K\to G_k\to 0\] defining the
inertia subgroup $I_K$. We let $\Frob_K=\Frob_k\in G_k$ be the
geometric Frobenius element, a topological generator of $G_k\cong\Zhat$. 

Then we define the Weil group $W_K$ via the commutative diagram
\[\xymatrix{0\ar[r] &I_K\ar[r]&G_K\ar[r]&G_k\ar[r]&0\\0\ar[r]
  &I_K\ar[r]\ar@{=}[u]&W_K\ar[r]\ar@{^{(}->}[u]&\Frob_k^\Z\ar[r]\ar@{^{(}->}[u]&0} \]
so that $W_K$ is the subgroup of $G_K$ consisting of elements which
map to an integral power of the Frobenius in $G_k$. The group $W_K$ is
a topological group, but its topology is not the subspace topology of
$G_K$; rather, the topology is determined by decreeing that $I_K$ is
open, and has its usual topology. 

Let $K^\nr=\overline{K}^{I_K}$ be the maximal unramified extension of
$K$, and let $K\tame=\cup_{(m,l)=1}K^\nr(\varpi_K^{1/m})$ be the
maximal tamely ramified extension. Then the wild inertia subgroup
$P_K:=\Gal(\overline{K}/K\tame)$ is the unique Sylow pro-$l$ subgroup
of $I_K$. Let $\zeta=(\zeta_m)_{(m,l)=1}$ be a compatible system of primitive
roots of unity (i.e.\ $\zeta_{ab}^a=\zeta_b$).  Then we have a
character \[t_\zeta:I_K/P_K\isoto\prod_{p\ne l}\Z_p,\] defined
by \[\frac{\sigma(\varpi_K^{1/m})}{\varpi_K^{1/m}}=\zeta_m^{(t_\zeta(\sigma)\pmod{m})}.\]
\begin{exercise}
  Any other compatible system of roots of unity is of the form
  $\zeta^u$ for some $u\in\prod_{p\ne l}\Z_p^\times$, and we have $t_{\zeta^u}=u^{-1}t_\zeta$.
\end{exercise}
If $\sigma\in W_K$, then
$t_\zeta(\sigma\tau\sigma^{-1})=\varepsilon(\sigma)t_\zeta(\tau)$, where
$\varepsilon$ is the cyclotomic character. We let $t_{\zeta,p}$ be the
composite of $t_\zeta$ and the projection to $\Z_p$.

Local class field theory is summarised in the following
statement. (See for example~\cite{MR546607} for this and the other
facts about class field theory recalled below.)
\begin{thm}
  Let $W_K^{\ab}$ denote the group $W_K/\overline{[W_K,W_K]}$. Then
  there are unique isomorphisms $\Art_K:K^\times\isoto W_K^{\ab}$ such that
  \begin{enumerate}
  \item if $K'/K$ is a finite extension, then $\Art_{K'}=\Art_K\circ
    N_{K'/K}$, and
  \item we have a commutative
    square \[\xymatrix{K^\times\ar[r]^{\Art_K}\ar@{->>}[d]^{\val_K}&W_K^{\ab}\ar@{->>}[d]\\\Z\ar[r]&\Frob_K^\Z}\]where
    the bottom arrow is the isomorphism sending $a\mapsto\Frob_K^a$.
  \end{enumerate}

\end{thm}The continuous irreducible representations of the group $W_K^{\ab}$ are
just the continuous characters of $W_K$, and local class field theory gives a
simple description of them, as representations of
$K^\times=\GL_1(K)$. The local Langlands correspondence for $\GL_n$
(see Section~\ref{subsec:loclanglands}) is a
kind of $n$-dimensional generalisation of this, giving a description of
certain representations of $\GL_n(K)$ in terms of the $n$-dimensional representations of $W_K$.

\begin{defn}Let $L$ be a field of characteristic $0$.  A
  \emph{representation} of $W_K$ over $L$ is a representation (on a
  finite-dimensional $L$-vector space) which is continuous if $L$ has
  the discrete topology (i.e.\ a representation with open kernel).

 A \emph{Weil--Deligne} representation of $W_K$ on a finite-dimensional
 $L$-vector space $V$ is a pair $(r,N)$ consisting of a representation
 $r:W_K\to\GL(V)$, and an endomorphism $N\in\End(V)$ such that for all
 $\sigma\in
 W_K$,\[r(\sigma)Nr(\sigma)^{-1}=(\#k)^{-v_K(\sigma)}N,\]where
 $v_K:W_K\to\Z$ is determined by $\sigma|_{K^{\nr}}=\Frob_K^{v_K(\sigma)}$.
\end{defn}
\begin{remark}
  \leavevmode\begin{enumerate}
  \item 

  Since $I_K$ is compact and open in $W_K$, if $r$ is a representation
  of $W_K$ then $r(I_K)$ is finite.
\item $N$ is necessarily nilpotent. 
  \end{enumerate}
\end{remark}
\begin{exercise}\label{ex: WD reps}
  \leavevmode\begin{enumerate}
  \item Show that if $(r,V)$ is a representation of $W_K$ and $m\ge 1$
    then the following defines a Weil--Deligne representation
    $\Sp_m(r)$ with underlying vector space $V^m$: we let $W_K$ act
    via \[r|\Art_K^{-1}|^{m-1}_K\oplus
    r|\Art_K^{-1}|_K^{m-2}\oplus\dots\oplus r,\] and let $N$ induce an
    isomorphism from $r|\Art_K^{-1}|_K^{i-1}$ to
    $r|\Art_K^{-1}|_K^i$ for each $i<m-1$, and be $0$ on  $r|\Art_K^{-1}|_K^{m-1}$.
  \item Show that every Weil--Deligne representation $(r,V)$ for which
    $r$ is semisimple is isomorphic to a
    direct sum of representations $\Sp_{m_i}(r_i)$.
  \item Show that if $(r,V,N)$ is a Weil--Deligne representation of
    $W_K$, and $K'/K$ is a finite extension, then $(r|_{W_{K'}},V,N)$
    is a Weil--Deligne representation of $W_{K'}$.
  \item Suppose that $r$ is a representation of $W_K$. Show that if
    $\sigma\in W_K$ then for some positive integer $n$, $r(\sigma^n)$
    is in the centre of $r(W_K)$.
  \item Assume further that $\sigma\notin I_K$. Show that for any
    $\tau\in W_K$ there exists $n\in\Z$ and $m>0$ such that $r(\sigma^n)=r(\tau^m)$.
  \item Show that for a representation $r$ of $W_K$, the following
    conditions are equivalent:
    \begin{enumerate}
    \item $r$ is semisimple.
    \item $r(\sigma)$ is semisimple for all $\sigma\in W_K$.
    \item $r(\sigma)$ is semisimple for some $\sigma\notin I_K$.
    \end{enumerate}
  \item \label{last point} Let $(r,N)$ be a Weil--Deligne representation of $W_K$. Set
    $\tilde{r}(\sigma)=r(\sigma)^\semis$, the semisimplification of
    $r(\sigma)$. Prove that $(\tilde{r},N)$ is also a Weil--Deligne
    representation of $W_K$.

  \end{enumerate}
\end{exercise}
\begin{defn}
  We say that a Weil--Deligne representation $(r,N)$ is
  \emph{Frobenius semisimple} if $r$ is semisimple. With notation as
  in Exercise~\ref{ex: WD reps}~\eqref{last point}, we say that $(\tilde{r},N)$ is the \emph{Frobenius
    semisimplification} of $(r,N)$.
\end{defn}
\begin{defn}
  If $L$ is an algebraic extension of $\Qp$, then we say
  that an element $A\in\GL_n(L)$ is \emph{bounded} if it has
  determinant in $\cO_L^\times$, and characteristic polynomial in $\cO_L[X]$.
\end{defn}
\begin{exercise}
  $A$ is bounded if and only if it stabilises an $\cO_L$-lattice in
  $L^n$.
\end{exercise}
\begin{defn}
  Let $L$ be an algebraic extension of $\Qp$. Then we say
  that $r$ is \emph{bounded} if $r(\sigma)$ is bounded for all
  $\sigma\in W_K$.
\end{defn}
\begin{exercise}
  Show $r$ is bounded if and only if $r(\sigma)$ is bounded for some
  $\sigma\notin I_K$.
\end{exercise}
The reason for all of these definitions is the following theorem,
which in practice gives us a rather concrete classification of the
$p$-adic representations of $G_K$.
\begin{prop}\label{prop: monodromy thm p not l}
  (Grothendieck's monodromy theorem) Suppose that $l\ne p$, that
  $K/\Ql$ is finite, and that $V$ is a finite-dimensional $L$-vector
  space, with $L$ an algebraic extension of $\Qp$. Fix $\varphi\in W_K$ a lift of $\Frob_K$ and a compatible system
 $(\zeta_m)_{(m,l)=1}$ of primitive roots of unity. If $\rho:G_K\to\GL(V)$ is a continuous
 representation then there is a finite extension $K'/K$ and a uniquely
 determined nilpotent $N\in\End(V)$ such that for all $\sigma\in
 I_{K'}$, \[\rho(\sigma)=\exp(Nt_{\zeta,p}(\sigma)).\] For all
 $\sigma\in W_K$, we have
 $\rho(\sigma)N\rho(\sigma)^{-1}=\#k^{-v_K(\sigma)}N$. In fact, we have
 an equivalence of categories $\WD=\WD_{\zeta,\varphi}$ from the category
 of continuous representations of $G_K$ on finite-dimensional
 $L$-vector spaces to the category of bounded Weil--Deligne
 representations on finite-dimensional $L$-vector spaces,
 taking \[\rho\mapsto (V,r,N),\
 r(\tau):=\rho(\tau)\exp(-t_{\zeta,p}(\varphi^{-v_K(\tau)}\tau)N).\] The
 functors $\WD_{\zeta',\varphi'}$ and $\WD_{\zeta,\varphi}$ are naturally isomorphic.
\end{prop}
\begin{remark}
  Note that since $N$ is nilpotent, the exponential here is just a
  polynomial - there are no convergence issues!
\end{remark}
The proof is contained in the following exercise.
\begin{exercise}
  \leavevmode\begin{enumerate}
  \item By Exercise~\ref{ex: existence of galois stable lattice} there
    is a $G_K$-stable  $\cO_L$-lattice
    $\Lambda\subset V$. Show that  if $G_{K'}$ is the kernel of the
    induced map $G_K\to\Aut(\Lambda/p\Lambda)$, then $K'/K$ is a
    finite extension, and $\rho(G_{K'})$ is pro-$p$. Show that
    $\rho|_{I_{K'}}$ factors through $t_{\zeta,p}:I_{K'}\to\Zp$.
  \item Choose $\sigma\in I_{K'}$ such that $t_{\zeta,p}(\sigma)$ topologically
    generates $t_{\zeta,p}(I_{K'})$. By considering the action of
    conjugation by $\varphi$, show that the eigenvalues of $\rho(\sigma)$
    are all $p$-power roots of unity. Hence show that one may make a
    further finite extension $K''/K'$ such that the elements of
    $\rho(I_{K''})$ are all unipotent.
  \item Deduce the existence of a unique nilpotent $N\in\End(V)$ such that for all $\sigma\in
 I_{K''}$, $\rho(\sigma)=\exp(Nt_{\zeta,p}(\sigma)).$ [Hint: use the
 logarithm map (why are there no convergence issues?).]
\item Complete the proof of the proposition, by showing that $(r,N)$
  is a Weil--Deligne representation. Where does the condition that $r$
  is bounded come in?
  \end{enumerate}
\end{exercise}
One significant advantage of Weil--Deligne representations over Galois
representations is that there are no subtle topological issues: the
topology on the Weil--Deligne representation is the discrete
topology. This allows one to describe representations in a way that is
``independent of $L$'', and is necessary to make sense of the notion
of a compatible system of Galois representations (or at least to make
sense of it at places at which the Galois representation is ramified);
see Definition~\ref{defn:strictly compatible system} below.
\subsection{Local representations with $p=l$: $p$-adic Hodge
  theory}\label{subsec: p adic HT}The case $l=p$ is far more complicated than the case $l\ne
p$, largely because wild inertia can act in a highly nontrivial
fashion, so there is no simple analogue of Grothendieck's monodromy
theorem. (There is still an analogue, though, it's just much harder to
state and prove, and doesn't apply to all $p$-adic Galois representations.) The study of representations $G_K\to\GL_n(\Qpbar)$
with $K/\Qp$ finite is part of what is called \emph{$p$-adic Hodge theory}, a subject
initially developed by Fontaine in the 1980s. 
For an introduction to the part of $p$-adic Hodge theory concerned
with Galois representations, the reader could consult
\cite{MR2023292}. There is a lot more to $p$-adic Hodge theory than
the study of Galois representations, and an excellent overview of some
recent developments in the more geometric part of the theory can be
found in~\cite{bhatt2021algebraic}.
We will content ourselves with some
terminology, some definitions, and some remarks intended to give
intuition and motivation.

Fix $K/\Qp$ finite. In some sense, ``most'' $p$-adic Galois representations
$G_K\to\GL_n(\Qpbar)$ will not be relevant for us, because they do not
arise in geometry, or in the Galois representations associated to
automorphic representations. Instead, there is a hierarchy of classes
of
representations \[\{\text{crystalline}\}\subsetneq\{\text{semistable}\}\subsetneq\{\text{de
  Rham}\}\subsetneq\{\text{Hodge--Tate}\}.\] For any of these classes $X$,
we say that $\rho$ is \emph{potentially} $X$ if there is a finite
extension $K'/K$ such that $\rho|_{G_{K'}}$ is $X$. A representation
is potentially de Rham if and only if it is de Rham, and potentially
Hodge--Tate if and only if it is Hodge--Tate; the corresponding
statements for crystalline and semistable representations are false,
as we will see concretely in the case $n=1$ later on. The $p$-adic
analogue of Grothendieck's monodromy theorem is the following deep
theorem of Berger.
\begin{thm}
  (The $p$-adic monodromy theorem) A representation is de Rham if and
  only if it is potentially semistable.
\end{thm}
The notion of a de Rham representation is designed to capture the
representations arising in geometry; it does so by the following
result of Tsuji (building on the work of many people).
\begin{thm}\label{thm: Tsuji}
  If $X/K$ is a smooth projective variety, then each
  $H^i_\et(X\times_K \Kbar,\Qpbar)$ is a de Rham representation.
\end{thm}
Similarly, the definitions of crystalline and semistable are designed
to capture the notions of good and semistable reduction, and one has
(again as a consequence of Tsuji's work; see Section 2.5
of~\cite{MR2023292} and the references therein)
\begin{thm}If $X/K$ is a smooth projective variety with good
  (respectively, semistable) reduction, then each
  $H^i_\et(X\times_K \Kbar,\Qpbar)$ is a crystalline (respectively,
  semistable) representation.
  \end{thm}
Thus the $p$-adic monodromy theorem can be thought of as a
Galois-theoretic incarnation of Grothendieck's semistable reduction
theorem. 

The case that $n=1$ is particularly simple, as we now explain. In this
case, every semistable character is crystalline, and the de Rham
characters are exactly the Hodge--Tate characters. In the case
$K=\Qp$, these are precisely the characters whose restrictions to
inertia are of the form
$\psi\varepsilon_p^m$ where $\psi$ has finite order and $m\in\Z$,
while the crystalline characters are those for which $\psi$ is
trivial. A similar description exists for general $K$, with
$\varepsilon_p^m$ replaced by a product of so-called \emph{Lubin--Tate
  characters}. 
\begin{fact}\label{fact: de Rham crystalline char loc alg}
 A character
  $\chi:G_K\to\Qpbartimes$ is de Rham if and only if there is an open
  subgroup $U$ of $K^\times$ and an integer $n_\tau$ for each
  $\tau:K\into\Qpbar$ such that
  $(\chi\circ\Art_K)(\alpha)=\prod_\tau\tau(\alpha)^{-n_\tau}$ for
  each $\alpha\in U$, and it is crystalline if and only if we can take
  $U=\cO_K^\times$.   (See Exercise 6.4.3 of~\cite{brinonconrad}.)
\end{fact}

As soon as $n>1$, there are non-crystalline semistable
representations, and non-de Rham Hodge--Tate representations. A useful
heuristic when comparing to the $l\ne p$ case is that crystalline
representations correspond to unramified representations, semistable
representations correspond to representations for which inertia acts
unipotently, and de Rham representations correspond to all
representations.

Suppose that $\rho:G_K\to\GL_n(\Qpbar)$ is a Hodge--Tate
representation. Then for each $\tau:K\into\Qpbar$ there is a multiset
of \emph{$\tau$-labeled Hodge--Tate weights} (defined for example in
the notation section of~\cite{blggt}, where they are called
``Hodge--Tate numbers'') $\HT_\tau(\rho)$
associated to $\rho$; this is a multiset of integers, and in the case
of a de Rham character $\chi$ as above,
$\HT_\tau(\chi)=n_\tau$. In particular,
the $p$-adic cyclotomic character $\varepsilon_p$ has all Hodge--Tate
weights equal to $-1$. If $K'/K$ is a finite extension, and
$\tau':K'\into\Qpbar$ extends $\tau:K\into\Qpbar$, then
$\HT_{\tau'}(\rho|_{G_{K'}})=\HT_\tau(\rho)$.

If furthermore $\rho$ is potentially semistable (equivalently, de
Rham) then a construction of Fontaine associates a Weil--Deligne
representation $\WD(\rho)=(r,N)$ of $W_K$ to $\rho$. If $K'/K$ is a
finite extension, then $\WD(\rho|_{G_{K'}})=(r|_{W_{K'}},N)$. It is
  known that $\rho$ is semistable if and only if $r$ is unramified,
  and that $\rho$ is crystalline if and only if $r$ is unramified and
  $N=0$. Thus $\rho$ is potentially crystalline if and only $N=0$.

\subsection{Number fields}We now consider the case that $K$ is a
number field (that is, a finite extension of $\Q$). If $v$ is a finite
place of $K$, we let $K_v$ denote the completion of $K$ at $v$. If
$K'/K$ is a finite Galois extension, then $\Gal(K'/K)$ transitively
permutes the places of $K'$ above $v$; if we choose one such place
$w$, then we define the \emph{decomposition
  group} \[\Gal(K'/K)_w:=\{\sigma\in\Gal(K'/K)|w\sigma=w\}.\] Then we
have a natural isomorphism $\Gal(K'/K)_w\isoto\Gal(K'_w/K_v)$, and
since $\Gal(K'/K)_{w\sigma}=\sigma^{-1}\Gal(K'/K)_w\sigma$, we see
that the definition extends to general algebraic extensions, and in
particular we have an embedding $G_{K_v}\into G_K$ which is well-defined up
to conjugacy (alternatively, up to a choice of embedding
$\Kbar\into\Kbar_v$). (Note that you need to be slightly careful with taking completions in the case
that $K'/K$ is infinite, as then the extension $K'_w/K_v$ need not be algebraic;
we can for example define $\Gal(K'_w/K_v)$ to be the group of continuous
automorphisms of $K'_w$ which fix $K_v$ pointwise.)

If $K'/K$ is Galois and unramified at $v$, and $w$ is a place
of $K'$ lying over $v$, then we
define \[\Frob_w:=\Frob_{K_v}\in\Gal(K'_w/K_v)\isoto\Gal(K'/K)_w\into\Gal(K'/K).\]
We have $\Frob_{w\sigma}=\sigma^{-1}\Frob_w\sigma$, and thus a
well-defined conjugacy class $[\Frob_v]=\{\Frob_w\}_{w|v}$ in $\Gal(K'/K)$.
\begin{fact}
  (Chebotarev density theorem) If $K'/K$ is a Galois extension which
  is unramified outside of a finite set $S$ of places of $K$, then the
  union of the conjugacy classes $[\Frob_v]$, $v\notin S$ is dense in $\Gal(K'/K)$.
\end{fact}

We briefly recall a statement of global class field theory. Let
$\A_K$ denote the adeles of $K$, and write
$K_\infty=\prod_{v|\infty}K_v$. Let $K^\ab=\Kbar^{[G_K,G_K]}$ be the
maximal abelian extension of $K$. Then there is a homomorphism
$\Art_K:\A^\times_K/(K_\infty^\times)^\circ\to\Gal(\Kab/K)$, defined
in the following way: for each finite place $v$ of $K$,  the
restriction of $\Art_K$ to $K_v^\times$ agrees with the local Artin
maps $\Art_{K_v}$, and similarly at the infinite places, it agrees
with the obvious isomorphisms
$\Art_{K_v}:K_v^\times/(K_v^\times)^\circ\isoto\Gal(\Kbar_v/K_v)$. (In
both cases, the symbol ${}^\circ$ refers to the connected component of
the identity.) Then
global class field theory states that $\Art_K$ induces an
isomorphism\[\Art_K:\A_K^\times/\overline{K^\times(K^\times_\infty)^\circ}\isoto\Gal(\Kab/K).\]

The global Galois representations that we will care about are those
that Fontaine and Mazur call \emph{geometric}. Let $L/\Qp$ be an
algebraic extension.
\begin{defn}
  If~$K$ is a number field, then a continuous representation $\rho:G_K\to\GL_n(L)$ is
  \emph{geometric} if it is unramified outside of a finite set of
  places of $K$, and if for each place $v|p$, $\rho|_{G_{K_v}}$ is de Rham.
\end{defn}
\begin{rem}
  It is known that both conditions are necessary; that is, there are
  examples of representations which are unramified outside of a finite
  set of places of $K$ but not de Rham at places lying over $p$, and
  examples of representations which are de Rham at all places lying
  over $p$, but are ramified at infinitely many places. (As we will
  see in Theorem~\ref{thm: Gr\"ossencharacters and algebraic
    representations}, these examples require $n>1$.)
\end{rem}
In practice (and conjecturally always), geometric Galois
representations arise as part of a \emph{compatible system} of Galois
representations. There are a number of different definitions of a
compatible system in the literature, all of which are conjecturally
equivalent (although proving the equivalence of the definitions is
probably very hard). The following definition, taken
from~\cite{blggt}, is simultaneously a strong enough set of assumptions under
which one can hope to employ automorphy lifting theorems to study a
compatible system, and is weak enough that the conditions can be
verified in interesting examples.
\begin{defn}
  \label{defn:weakly compatible system}Suppose that $K$ and $M$ are
number fields, that $S$ is a finite set of places of $K$ and that $n$
is a positive integer. By a {\em weakly compatible system} of
$n$-dimensional $p$-adic representations (for varying $p$) of $G_K$ defined over $M$ and unramified outside $S$ we mean a family
of continuous semisimple representations
\[ r_\lambda: G_K \lra \GL_n(\barM_\lambda), \]
where $\lambda$ runs over the finite places of $M$, with the following properties.
\begin{itemize}
\item If $v\notin S$ is a finite place of $K$, then for all $\lambda$ not dividing the residue characteristic of $v$, the representation $r_\lambda$ is unramified at $v$ and the characteristic
polynomial of $r_\lambda(\Frob_v)$ lies in $M[X]$ and is independent of $\lambda$. 
\item Each
representation $r_\lambda$ is de Rham at all places above the residue characteristic of $\lambda$, and in fact crystalline at any place $v \not\in S$ which divides
the residue characteristic of $\lambda$. 
\item For each embedding $\tau:K \into \barM$ the $\tau$-labeled Hodge--Tate weights of $r_\lambda$ are independent of~ $\lambda$.
\end{itemize}
\end{defn}
\begin{rem}
  By the Chebotarev density theorem and the Brauer--Nesbitt theorem, each $r_\lambda$ is determined by
  the characteristic polynomials of the $r_\lambda(\Frob_v)$ for
  $v\notin S$, and in particular the compatible system is determined
  by a single $r_\lambda$. Note that for a general element $\sigma\in G_K$,
  there will be no relationship between the characteristic polynomials
  of the $r_\lambda(\sigma)$ as $\lambda$ varies (and they won't even
  lie in $M[X]$, so there will be no way of comparing them).
\end{rem}

There are various other properties one could demand; for example, we
have the following definition (again following~\cite{blggt}, although
we have slightly strengthened the definition made there by allowing
$\lambda$ to divide the residue characteristic of $v$).
\begin{defn}
  \label{defn:strictly compatible system}We say that a weakly
  compatible system is {\em strictly
    compatible} if for each finite place $v$ of $K$ there is a
  Weil--Deligne representation $\WD_v$ of $W_{K_v}$ over $\barM$
  such that for each finite place $\lambda$ of $M$ and every $M$-linear embedding
  $\varsigma:\barM \into \barM_\lambda$ we have
  $\varsigma\WD_v\cong \WD(r_\lambda|_{G_{K_v}})^\Fsemis$.
\end{defn}
Conjecturally, every weakly compatible system is strictly compatible,
and even satisfies further properties, such as purity (see e.g.\ Section 5
of~\cite{blggt}). We also have the following consequence of the
Fontaine--Mazur conjecture (Conjecture~\ref{conj:FM conj} below) and standard conjectures on the \'etale
cohomology of algebraic varieties over number fields.
\begin{conj}\label{conj: geometric implies strictly compatible}
  Any semisimple geometric representation $G_K\to\GL_n(L)$ is part of a strictly
  compatible system of Galois representations.
\end{conj}
In practice, most progress on understanding these conjectures has been
made by using automorphy lifting theorems to prove special cases of
the following conjecture.
\begin{conj}\label{conj: strictly compatible implies automorphic}
  Any weakly compatible system of Galois representations is strictly
  compatible, and is in addition automorphic, in the sense that there
  is an algebraic automorphic representation (in the sense of
  \cite{MR1044819}) $\pi$ of $\GL_n(\A_K)$ with the property that
  $\WD_v\cong\rec_{K_v}(\pi_v|\det|^{(1-n)/2})$ for each finite place
  $v$ of $K$, where $\rec_{K_v}$ is the local Langlands correspondence as in
  Section~\ref{subsec:loclanglands} below.
\end{conj}

\subsection{Sources of Galois representations}\label{subsec:sources of
  Galois repns}
The main source (and conjecturally the only source) of compatible
systems of Galois representations is the \'etale cohomology of
algebraic varieties. We have the following result, a consequence of
Theorem~\ref{thm: Tsuji} and ``independence of $l$'' results in \'etale cohomology~\cite{MR332791}.
\begin{thm}\label{thm:etale cohomology gives a compatible system}
  Let $K$ be a number field, and let $X/K$ be a smooth projective
  variety. Then for any $i,j$, the
  $H^i_{\et}(X\times_K\Kbar,\Qp)^\semis(j)$ (the $(j)$ denoting a Tate
  twist) form a weakly compatible system  (defined over~$\Q$) as~$p$ varies.
\end{thm}
\begin{rem}\label{rem: conj that etale cohomology is strictly
    compatible} Conjecturally, it is a strictly compatible system, and
  there is no need to semisimplify the representations. Both of these
  properties are known if $X$ is an abelian variety (see Section 2.4 of~\cite{MR1293977}).
\end{rem}
\begin{conj}\label{conj:FM conj}
  (The Fontaine--Mazur conjecture,~\cite{MR1363495}) Any irreducible
  geometric representation $\rho:G_K\to\GL_n(\Qpbar)$ is (the
  extension of scalars to $\Qpbar$ of) a subquotient
  of a representation arising from \'etale cohomology as in
  Theorem~\ref{thm:etale cohomology gives a compatible system}. 
\end{conj}
\begin{rem}
  The \emph{Fontaine--Mazur--Langlands conjecture} is a somewhat
  ill-defined conjecture, which is essentially the union of
  Conjectures~\ref{conj: geometric implies strictly compatible}
  and~\ref{conj: strictly compatible implies automorphic}, expressing
  the expectation that an irreducible geometric Galois representation
  is automorphic.
\end{rem}
When $n=1$, all of these conjectures are essentially known, as we will
now explain. For $n>1$, we know very little (although the situation
when $K=\Q$ and $n=2$ is pretty good), and the main results that are
known are as a consequence of automorphy lifting theorems (as
discussed in these notes) and of potential automorphy theorems (which
are not discussed in these notes, but should be accessible given the
material we develop here; for a nice introduction,
see~\cite{MR2905534}).

\begin{defn}
  \label{defn: grossenchar}A \emph{Gr\"ossencharacter} is a continuous
  character $\chi:\A_K^\times/K^\times\to\C^\times$. We say that
  $\chi$ is \emph{algebraic} (or ``type $A_0$'') if for each
  $\tau:K\into\C$ there is an integer $n_\tau$, such that for each
  $\alpha\in(K_\infty^\times)^\circ$, we have
  $\chi(\alpha)=\prod_\tau(\tau(\alpha))^{-n_\tau}$.
\end{defn}

\begin{defn}
  \label{defn: algebraic char}Let $L$ be a field of characteristic
  zero such that for each embedding $\tau:K\into\Lbar$, we have
  $\tau(K)\subseteq L$. Then an \emph{algebraic character}
  $\chi_0:\A_K^\times\to \Lbar^\times$ is a character with open kernel
  such that for each $\tau:K\into L$ there is an integer $n_\tau$ with
  the property that for all $\alpha\in K^\times$, we have
  $\chi_0(\alpha)=\prod_\tau(\tau(\alpha))^{n_\tau}$. 
\end{defn}
\begin{exercise}
  Show that if $\chi_0$ is an algebraic character, then $\chi_0$ takes
  values in some number field. [Hint: show that
  $\A_K^\times/(K^\times\ker\chi_0)$ is finite, and that
  $\chi_0(K^\times\ker\chi_0)$ is contained in a number field.]
\end{exercise}
\begin{thm}\label{thm: Gr\"ossencharacters and algebraic
    representations}Let $E$ be a number field containing the normal
  closure of $K$. Fix embeddings $\imath_\infty:\Ebar\into\C$,
  $\imath_p:\Ebar\into\Qpbar$. Then the following are in natural
  bijection.
  \begin{enumerate}
  \item Algebraic characters $\chi_0:\A_K^\times\to\Ebar^\times$.
  \item Algebraic Gr\"ossencharacters $\chi:\A_K^\times/K^\times\to\C^\times$.
  \item Continuous representations $\rho:G_K\to\Qpbartimes$ which
    are de Rham at all $v|p$.
  \item Geometric representations $\rho:G_K\to\Qpbartimes$.
  \end{enumerate}

  \end{thm}
  \begin{exercise}
    Prove Theorem~\ref{thm: Gr\"ossencharacters and algebraic
      representations} as follows (see e.g.\ Section 1
    of~\cite{farguesnotes} for more details). Firstly, use 
Fact~\ref{fact: de Rham crystalline char loc alg}, together with global class field
    theory, to show that (3) and (4) are equivalent. For the correspondence between (1) and (2), show that we can pair up $\chi_0$ and
$\chi$
by \[\chi(\alpha)=\imath_\infty\left(\chi_0(\alpha)\prod_{\tau:K\into\C}\tau(\alpha_\infty)^{-n_{\imath_\infty^{-1}\tau}}\right).\]
For the correspondence between (1) and (3), show that we can pair up
$\chi_0$ and $\rho$ by \[(\rho\circ\Art_K)(\alpha)=\imath_p\left(\chi_0(\alpha)\prod_{\tau:K\into\Qpbar}\tau(\alpha_p)^{-n_{\imath_p^{-1}\tau}}\right).\]
  \end{exercise}

  \section{Galois deformations}The ``lifting'' in ``modularity lifting
  theorems'' refers to deducing the modularity of a $p$-adic Galois
  representation from the modularity of its reduction modulo $p$; so
  we ``lift'' the modularity property from characteristic~$p$ to
  characteristic zero. In this section we consider the
  Galois-theoretic aspects of this lifting, which are usually known as
  ``Galois deformation theory''.

  There are a number of good introductions
to the material in this section, and for the most part we will simply
give basic definitions and motivation, and refer elsewhere for
proofs. In particular, \cite{MR1638481} is a very nice introduction to
Galois deformations (although slightly out of date, as it does not
treat liftings/framed deformations), and \cite{boecklenotes} is a
thorough modern treatment.
\subsection{Generalities}Take $L/\Qp$ finite with ring of integers
$\cO=\cO_L$ and maximal ideal $\lambda$, and write
$\F=\cO/\lambda$. Let $G$ be a profinite group which satisfies the
following condition (Mazur's condition $\Phi_p$): for each
open subgroup~$\Delta$ of $G$, then
$\Delta/\langle[\Delta,\Delta],\Delta^p\rangle$ is
finite. Equivalently (see e.g.\ Exercise 1.8.1 of~\cite{boecklenotes}), for
each $\Delta$ the maximal pro-$p$ quotient of $\Delta$ is
topologically finitely generated. If $G$ is topologically finitely
generated, then $\Phi_p$ holds, but we will need to use the condition
for some $G$ (the global Galois groups $G_{K,S}$ defined below) which
are not known to be topologically finitely generated.

In particular, using class field theory or Kummer theory, it can be
checked that $\Phi_p$ holds if $G=G_K=\Gal(\overline{K}/K)$ for some
prime $l$ (possibly equal to $p$) and some finite extension $K/\Ql$,
or if $G=G_{K,S}=\Gal(K_S/K)$ where $K$ is a number field, $S$ is a
finite set of finite places of $K$, and $K_S/K$ is the maximal extension
unramified outside of $S$ and the infinite places (see e.g.\ the proof of Theorem 2.41
of~\cite{MR1605752}).

Fix a continuous representation $\rhobar:G\to\GL_n(\F)$. Let $\cC_\cO$ be the
category of complete local Noetherian $\cO$-algebras with residue
field $\F$, and consider the functor $\cC_\cO\to\Sets$ which sends $A$ to
the set of continuous representations $\rho:G\to\GL_n(A)$ such that
$\rho\text{ mod }\m_A=\rhobar$ (that is, to the set of \emph{lifts} of
$\rhobar$ to $A$).
\begin{lem}
  This functor is represented by a representation
  $\rho^\square:G\to\GL_n(R_{\rhobar}^\square)$. 
\end{lem}
\begin{proof}
  This is straightforward; see Proposition 1.3.1(a)
  of~\cite{boecklenotes} for a closely related result (showing the
  prorepresentability of the functor restricted to Artinian algebras),
  or Dickinson's appendix to~\cite{MR1860043} for a complete proof of
  a more general result.
\end{proof}

\begin{defn}
  We say that
  $R_{\rhobar}^\square$ is the \emph{universal lifting ring} (or in
  Kisin's terminology, the \emph{universal framed deformation
    ring}). We say that $\rho^\square$ is the \emph{universal lifting}
  of $\rhobar$.
\end{defn}
If $\End_{\F[G]}\rhobar=\F$ we will say that $\rhobar$ is
\emph{Schur}. By Schur's lemma, if $\rhobar$ is absolutely
irreducible, then $\rhobar$ is Schur. In this case, there is a very
useful (and historically earlier) variant on the above construction.
\begin{defn}
  Suppose that $\rhobar$ is Schur. Then a \emph{deformation} of
  $\rhobar$ to $A\in\ob\cC_\cO$ is an equivalence class of liftings,
  where $\rho\sim\rho'$ if and only if $\rho'=a\rho a^{-1}$ for some
  $a\in\ker(\GL_n(A)\to\GL_n(\F))$ (or equivalently, for some $a\in\GL_n(A)$).
\end{defn}
\begin{lem}
  If $\rhobar$ is Schur, then the functor $\cC_\cO\to\Sets$ sending
  $A$ to the set of deformations of $\rhobar$ to $A$ is representable
  by some $\rho^\univ:G\to\GL_n(R_{\rhobar}^\univ)$. 
\end{lem}
\begin{proof}
  See  Proposition 1.3.1(b) of~\cite{boecklenotes}, or Theorem 2.36
  of~\cite{MR1605752} for a more hands-on approach.
\end{proof}
\begin{defn}
  We say that $\rho^\univ$ (or more properly, its equivalence class)
  is the \emph{universal deformation} of
  $\rhobar$, and $R_{\rhobar}^\univ$ is the \emph{universal
    deformation ring}.
\end{defn}

Deformations are representations considered up to conjugation, so it
is reasonable to hope that deformations can be studied by considering
their traces. In the case that $\rhobar$ is absolutely irreducible, universal
deformations are determined by traces in the following rather strong
sense. This result is essentially due to Carayol \cite{MR1279611}.
\begin{lem}\label{lem: carayol traces lemma}
  Suppose that $\rhobar$ is absolutely irreducible. Let $R$ be an
  object of $\cC_\cO$, and $\rho:G\to\GL_n(R)$ a lifting of
  $\rhobar$.
  \begin{enumerate}
  \item If $a\in\GL_n(R)$ and $a\rho a^{-1}=\rho$ then $a\in R^\times$.
  \item If $\rho':G\to\GL_n(R)$ is another continuous lifting of
    $\rhobar$ and $\tr\rho=\tr\rho'$, then there is some
    $a\in\ker(\GL_n(R)\to\GL_n(\F))$ such that $\rho'=a\rho a^{-1}$.
  \item If $S\subseteq R$ is a closed subring with $S\in\ob\cC_\cO$ and
    $\m_S=\m_R\cap S$, and if $\tr\rho(G)\subseteq S$, then  there is
    some $a\in\ker(\GL_n(R)\to\GL_n(\F))$ such that $a\rho a^{-1}:G\to\GL_n(S)$.
  \end{enumerate}
\end{lem}
\begin{proof}See Lemmas 2.1.8 and 2.1.10 of~\cite{cht}, or Theorem
  2.2.1 of~\cite{boecklenotes}.
  \end{proof}
\begin{exercise}\label{exercise: deformation ring generated by traces}
  Deduce from Lemma~\ref{lem: carayol traces lemma} that if $\rhobar$
  is absolutely irreducible, then $R_{\rhobar}^\univ$ is topologically
  generated over $\cO$ by the values $\tr\rho^\univ(g)$ as $g$ runs
  over any dense subset of $G$.
\end{exercise}
\begin{exercise}
  Show that if $\rhobar$ is absolutely irreducible, then
  $R_{\rhobar}^\square$ is isomorphic to a power series ring in
  $(n^2-1)$ variables over $R_{\rhobar}^\univ$. Hint: let $\rho^\univ$
  be a choice of universal deformation, and consider the
  homomorphism \[\rho^\square:G\to
  \GL_n(R_{\rhobar}^\univ\llbracket X_{i,j}\rrbracket _{i,j=1,\dots,n}/(X_{1,1}))\] given by
  $\rho^\square=(1_n+(X_{i,j}))\rho^\univ(1_n+(X_{i,j}))^{-1}$. Show that
  this is the universal lifting.
\end{exercise}
\subsection{Tangent spaces}The tangent spaces of universal lifting and
deformation rings have a natural interpretation in terms of liftings and
deformations to the ring of dual numbers,
$\F[\varepsilon]/(\varepsilon^2)$.
\begin{exercise}
  Show that we have natural bijections between
  \begin{enumerate}
  \item $\Hom_{\F}(\m_{R^\square_{\rhobar}}/(\m_{R^\square_{\rhobar}}^2,\lambda),{\F})$.
  \item $\Hom_{\cO}(R^\square_{\rhobar},{\F}[\varepsilon]/(\varepsilon^2))$.
  \item The set of liftings of $\rhobar$ to ${\F}[\varepsilon]/(\varepsilon^2)$.
  \item The set of cocycles $Z^1(G,\ad\rhobar)$.
  \end{enumerate}
Show that if $\rhobar$ is absolutely irreducible, then we also have a
bijection between
$\Hom_{\F}(\m_{R^\univ_{\rhobar}}/(\m_{R^\univ_{\rhobar}}^2,\lambda),{\F})$
and $H^1(G,\ad\rhobar)$.

Hint: given
$f\in\Hom_{\F}(\m_{R^\square_{\rhobar}}/(\m_{R^\square_{\rhobar}}^2,\lambda),{\F})$,
define an element of
$\Hom_{\cO}(R^\square_{\rhobar},{\F}[\varepsilon]/(\varepsilon^2))$ by sending
$a+x$ to $a+f(x)\varepsilon$ whenever $a\in\cO$ and
$x\in\m_{R^\square_{\rhobar}}$. Given a cocycle
$\phi\in Z^1(G,\ad\rhobar)$, define a lifting
$\rho:G\to\GL_n({\F}[\varepsilon]/(\varepsilon^2))$ by
$\rho(g):=(1+\phi(g)\varepsilon)\rhobar(g)$. 
\end{exercise}
\begin{cor}
  We have $\dim_{\F}
  \m_{R^\square_{\rhobar}}/(\m_{R^\square_{\rhobar}}^2,\lambda)=\dim_{\F}H^1(G,\ad\rhobar)+n^2-\dim_{\F} H^0(G,\ad\rhobar)$.
\end{cor}
\begin{proof}
  This follows from the exact sequence
  \[0\to(\ad\rhobar)^G\to\ad\rhobar\to Z^1(G,\ad\rhobar)\to
  H^1(G,\ad\rhobar)\to 0.\qedhere\]
\end{proof}
In particular, if $d:=\dim_{\F}Z^1(G,\ad\rhobar)$, then we can choose a
surjection $\phi:\cO\llbracket x_1,\dots,x_d\rrbracket \onto
R_{\rhobar}^\square$. Similarly, if $\rhobar$ is absolutely
irreducible, we can choose a surjection $\phi':\cO\llbracket x_1,\dots,x_{d'}\rrbracket \onto
R_{\rhobar}^\univ$, where $d':=\dim_{\F} H^1(G,\ad\rhobar)$.
\begin{lem}
  If $J=\ker\phi$ or $J=\ker\phi'$, then there is an injection
  $\Hom_\F(J/\m J,\F)\into H^2(G,\ad\rhobar)$, where $\m$ denotes the
  maximal ideal of $\cO\llbracket x_1,\dots,x_d\rrbracket $ or
  $\cO\llbracket x_1,\dots,x_{d'}\rrbracket $ respectively.
\end{lem}
\begin{proof}
  See the proof of Proposition 2 of~\cite{MR1012172}.
\end{proof}
\begin{cor}\label{cor:dimension of unrestricted deformation rings}
If $H^2(G,\ad\rhobar)=0$, then $R_{\rhobar}^\square$ is formally
smooth of relative dimension $\dim_{\F} Z^1(G,\ad\rhobar)$ over $\cO$.

  In any case, the Krull dimension of $R_{\rhobar}^\square$ is at
  least \[1+n^2-\dim_{\F} H^0(G,\ad\rhobar)+\dim_{\F} H^1(G,\ad\rhobar)
  -\dim_{\F} H^2(G,\ad\rhobar).\]If $\rhobar$ is absolutely irreducible,
  then the Krull dimension of $R_{\rhobar}^\univ$ is at
  least \[1+\dim_{\F} H^1(G,\ad\rhobar)-\dim_{\F} H^2(G,\ad\rhobar).\]
\end{cor}
\subsection{Deformation conditions}In practice, we frequently want to
impose additional conditions on the liftings and deformations we
consider. For example, if we are trying to prove the Fontaine--Mazur
conjecture, we would like to be able to restrict to global deformations
which are geometric. There are various ways in
which to impose extra conditions; we will use the formalism of
\emph{deformation problems} introduced in \cite{cht}.
\begin{defn}
  By a \emph{deformation problem} $\cD$ we mean a collection of
  liftings $(R,\rho)$ of $(\F,\rhobar)$ (with $R$ an object of
  $\cC_\cO$), satisfying the following properties.
  \begin{itemize}
  \item $(\F,\rhobar)\in\cD$.
  \item If $f:R\to S$ is a morphism in $\cC_\cO$ and $(R,\rho)\in\cD$,
    then $(S,f\circ\rho)\in\cD$.
  \item If $f:R\into S$ is an injective morphism in $\cC_\cO$ then $(R,\rho)\in\cD$
    if and only if $(S,f\circ\rho)\in\cD$.
  \item Suppose that $R_1,R_2\in\ob\cC_\cO$ and $I_1,I_2$ are closed ideals
    of 
    $R_1,R_2$ respectively such that there is an isomorphism
    $f:R_1/I_1\isoto R_2/I_2$. Suppose also that
    $(R_1,\rho_1),(R_2,\rho_2)\in\cD$, and that $f(\rho_1\text{ mod
    }I_1)=\rho_2\text{ mod }I_2$.

Then $(\{(a,b)\in R_1\oplus R_2: f(a\text{ mod }I_1)=b\text{ mod }I_2\},\rho_1\oplus\rho_2)\in\cD$.
\item If $(R,\rho)$ is a lifting of $(\F,\rhobar)$ and $I_1\supset
  I_2\supset\cdots$ is a sequence of ideals of $R$ with
  $\cap_jI_j=0$, and $(R/I_j,\rho\text{ mod }I_j)\in\cD$ for all
  $j$, then $(R,\rho)\in\cD$.
\item If $(R,\rho)\in\cD$ and $a\in\ker(\GL_n(R)\to\GL_n(\F))$, then
  $(R,a\rho a^{-1})\in\cD$.
  \end{itemize}

\end{defn}
In practice, when we want to impose a condition on our deformations,
it will be easy to see that it satisfies these requirements. (An
exception is that these properties are hard to check for certain
conditions arising in $p$-adic Hodge theory, but we won't need those
conditions in these notes.)

The relationship of this definition to the universal lifting ring is
as follows. Note that each element
$a\in\ker(\GL_n(R_{\rhobar}^\square)\to\GL_n(\F))$ acts on
$R_{\rhobar}^\square$, via the universal property and by
sending $\rho^\square$ to $a^{-1}\rho^\square a$. [Warning: this
\emph{isn't} a group action, though!]
\begin{lem}
  (1) If $\cD$ is a deformation problem then there is a
  $\ker(\GL_n(R_{\rhobar}^\square)\to\GL_n(\F))$-invariant ideal
  $I(\cD)$ of $R_{\rhobar}^\square$ such that $(R,\rho)\in\cD$ if and
  only if the map $R_{\rhobar}^\square\to R$ induced by $\rho$ factors
  through the quotient $R_{\rhobar}^\square/I(\cD)$.

(2) Let $\tL(\cD)\subseteq
Z^1(G,\ad\rhobar)\cong\Hom(\m_{R_{\rhobar}^\square}/(\lambda,\m_{R_{\rhobar}^\square}^2),\F)$
denote the annihilator of the image of $I(\cD)$ in
$\m_{R_{\rhobar}^\square}/(\lambda,\m_{R_{\rhobar}^\square}^2)$.

Then $\tL(\cD)$ is the preimage of some subspace $L(\cD)\subseteq
H^1(G,\ad\rhobar)$.

(3) If $I$ is a
$\ker(\GL_n(R_{\rhobar}^\square)\to\GL_n(\F))$-invariant ideal of
$R_{\rhobar}^\square$ with $\sqrt{I}=I$ and
$I\ne\m_{R_{\rhobar}^\square}$, then
\[\cD(I):=\{(R,\rho):R_{\rhobar}^\square\to R\text{ factors through
}R_{\rhobar}^\square/I\}\] is a deformation problem. Furthermore, we
have $I(\cD(I))=I$ and $\cD(I(\cD))=\cD$.
\end{lem}
\begin{proof}
  See Lemma 2.2.3 of~\cite{cht} and Lemma 3.2 of~\cite{blght} (and for
  (2), use that $I(\cD)$ is $\ker(\GL_n(R_{\rhobar}^\square)\to\GL_n(\F))$-invariant).
\end{proof}
\subsection{Fixing determinants}For technical reasons, we will want to fix the
determinants of our Galois representations (see Remark 5.12 of~\cite{CalGer}). 
To this end, let
$\chi:G\to\cO^\times$ be a continuous homomorphism such that
$\chi\text{ mod }\lambda=\det\rhobar$. Then it makes sense to ask that
a lifting has determinant $\chi$, and we can define a universal
lifting ring $R_{\rhobar,\chi}^\square$ for lifts with determinant
$\chi$, and when $\rhobar$ is Schur, a universal fixed determinant
deformation ring $R_{\rhobar,\chi}^\univ$.
\begin{exercise}
  Check that the material developed in the previous section goes over
  unchanged, except that $\ad\rhobar$ needs to be replaced with
  $\ad^0\rhobar:=\{x\in\ad\rhobar:\tr x =0\}$. 
\end{exercise}
Note that since we are assuming throughout that $p\nmid n$, $\ad^0\rhobar$ is a direct summand of $\ad\rhobar$ (as a $G$-representation).
\subsection{Global deformations with local conditions}Now fix a finite
set $S$, and for each $v\in S$, a profinite group $G_v$ satisfying
$\Phi_p$, together with a continuous homomorphism $G_v\to G$, and a
deformation problem $\cD_v$ for $\rhobar|_{G_v}$. [In applications,
$G$ will be a global Galois group, and the $G_v$ will be decomposition
groups at finite places.]

Also fix $\chi:G\to\cO^\times$, a continuous homomorphism such that
$\chi\text{ mod }\lambda=\det\rhobar$. Assume that $\rhobar$ is
absolutely irreducible, and fix some subset $T\subseteq S$.
\begin{defn}
  Fix $A\in\ob\cC_\cO$. A {\em $T$-framed deformation} of $\rhobar$ of {\em type}
  $\cS:=(S,\{\cD_v\}_{v\in S},\chi)$ to $A$ is an equivalence class of
  tuples $(\rho,\{\alpha_v\}_{v\in T})$, where $\rho:G\to\GL_n(A)$ is a
  lift of $\rhobar$ such that $\det\rho=\chi$ and
  $\rho|_{G_v}\in\cD_v$ for all $v\in S$, and $\alpha_v$ is an element
  of $\ker(\GL_n(A)\to\GL_n(\F))$. 

  The equivalence relation is defined by decreeing that for each
  $\beta\in\ker(\GL_n(A)\to\GL_n(\F))$, we have
  $(\rho,\{\alpha_v\}_{v\in
    T})\sim(\beta\rho\beta^{-1},\{\beta\alpha_v\}_{v\in T})$.
\end{defn}
The point of considering $T$-framed deformations is that it allows us
to study absolutely irreducible representations $\rhobar$ for which
some of the $\rhobar|_{G_v}$ are reducible, because if $(\rho,\{\alpha_v\}_{v\in
    T})$ is a $T$-framed deformation of type $\cS$, then
  $\alpha_v^{-1}\rho|_{G_v}\alpha_v$ is a well-defined element of
  $\cD_v$ (independent of the choice of representative of the
  equivalence class). The following lemma should be unsurprising.
  \begin{lem}
    The functor $\cC_\cO\to\Sets$ sending $A$ to the set of $T$-framed
    deformations of $\rhobar$ of type $\cS$ is represented by a
    universal object $\rho^{\square_T}:G\to\GL_n(R_{\cS}^{\square_T})$.
  \end{lem}
  \begin{proof}
    See Proposition 2.2.9 of~\cite{cht}.
  \end{proof}
If $T=\emptyset$ then we will write $R_{\cS}^\univ$ for $R_{\cS}^{\square_T}$.
  \subsection{Presenting global deformation rings over local lifting
    rings}Continue to use the notation of the previous
  subsection. Since $\alpha_v^{-1}\rho^{\square_T}|_{G_v}\alpha_v$ is a
  well-defined element of $\cD_v$, we have a tautological homomorphism
  $R_{\rhobar|_{G_v},\chi}^\square/I(\cD_v)\to R_{\cS}^{\square_T}$. Define \[R_{\cS,T}^\loc:=\widehat{\otimes}_{v\in
    T}\left(R_{\rhobar|_{G_v},\chi}^\square/I(\cD_v)\right).\]Then we
have a natural map $R_{\cS,T}^\loc\to R_{\cS}^{\square_T}$. 

We now generalise Corollary~\ref{cor:dimension of unrestricted
  deformation rings} by considering presentations of
$R_{\cS}^{\square_T}$ over $R_{\cS,T}^\loc$. In order to compute how
many variables are needed to present $R_{\cS}^{\square_T}$ over
$R_{\cS,T}^\loc$, we must compute $\dim_{\F}
  \m_{R_{\cS}^{\square_T}}/(\m_{R_{\cS}^{\square_T}}^2,\m_{R_{\cS,T}^\loc},\lambda)$. Unsurprisingly,
  in order to compute this, we will compute a certain $H^1$.

We define a complex as follows. As usual, given a group $G$ and an
${\F}[G]$-module $M$, we let $C^i(G,M)$ be the space of functions $G^i\to
M$, and we let $\partial:C^i(G,M)\to C^{i+1}(G,M)$ be the usual coboundary
map. We define a complex $C^i_{\cS,T,\loc}(G,\ad^0\rhobar)$
by \[C^0_{\cS,T,\loc}(G,\ad^0\rhobar)=\oplus_{v\in
  T}C^0(G_{v},\ad\rhobar)\oplus\oplus_{v\in S\setminus T}0,\] \[C^1_{\cS,T,\loc}(G,\ad^0\rhobar)=\oplus_{v\in
  T}C^1(G_v,\ad^0\rhobar)\oplus\oplus_{v\in S\setminus
  T}C^1(G_v,\ad^0\rhobar)/\tL(\cD_v),\] and for $i\ge 2$, \[C^i_{\cS,T,\loc}(G,\ad^0\rhobar)=\oplus_{v\in
   S}C^i(G_v,\ad^0\rhobar)
.\]

Let $C^0_0(G,\ad^0\rhobar):=C^0(G,\ad\rhobar)$, and set
$C^i_0(G,\ad^0\rhobar)=C^i(G,\ad^0\rhobar)$ for $i>0$. Then we let $H^i_{\cS,T}(G,\ad^0\rhobar)$ denote the cohomology
of the
complex \[C^i_{\cS,T}(G,\ad^0\rhobar):=C_0^i(G,\ad^0\rhobar)\oplus C^{i-1}_{\cS,T,\loc}(G,\ad^0\rhobar)\]where the coboundary
  map is given by \[(\phi,(\psi_v))\mapsto(\partial
  \phi,(\phi|_{G_v}-\partial\psi_v)).\] 
  Then we have an exact sequence of complexes \[0\to C^{i-1}_{\cS,T,\loc}(G,\ad^0\rhobar) \to
C^i_{\cS,T}(G,\ad^0\rhobar)\to C^i_0(G,\ad^0\rhobar)\to 0,\]and the
corresponding long exact sequence in cohomology is

\[\begin{tikzpicture}[descr/.style={fill=white,inner sep=1.5pt}]
        \matrix (m) [
            matrix of math nodes,
            row sep=1em,
            column sep=2.5em,
            text height=1.5ex, text depth=0.25ex
        ]
        { 0 & H^0_{\cS,T}(G,\ad^0\rhobar) & H^0(G,\ad\rhobar) &
          \oplus_{v\in T}H^0(G_v,\ad\rhobar) \\
            & H^1_{\cS,T}(G,\ad^0\rhobar) & H^1(G,\ad^0\rhobar) &
          \oplus_{v\in T}H^1(G_v,\ad^0\rhobar)\oplus_{v\in S\setminus T}H^1(G_v,\ad^0\rhobar)/L(\cD_v) \\
            & H^2_{\cS,T}(G,\ad^0\rhobar) & H^2(G,\ad^0\rhobar) &
          \oplus_{v\in S}H^2(G_v,\ad^0\rhobar)  
          \\
                       &H^3_{\cS,T}(G,\ad^0\rhobar) &  \ldots\ldots\ldots\\
        };

        \path[overlay,->, font=\scriptsize,>=latex]
        (m-1-1) edge (m-1-2)
        (m-1-2) edge (m-1-3)
        (m-1-3) edge (m-1-4)
        (m-1-4) edge[out=355,in=175]  (m-2-2)
        (m-2-2) edge (m-2-3)
        (m-2-3) edge (m-2-4)
        (m-2-4) edge[out=355,in=175]  (m-3-2)
        (m-3-2) edge (m-3-3)
        (m-3-3) edge (m-3-4)
        (m-3-4) edge[out=355,in=175] (m-4-2)
        (m-4-2) edge (m-4-3);
\end{tikzpicture}\]
Taking Euler characteristics, we see
that if we define the negative Euler characteristic~ $\chi$ by $\chi(G,\ad^0\rhobar)=\sum_i(-1)^{i-1}\dim_{\F}
H^i(G,\ad^0\rhobar),$ we
have \begin{align*}\chi_{\cS,T}(G,\ad^0\rhobar)=&-1+\chi(G,\ad^0\rhobar)-\sum_{v\in
    S}\chi(G_v,\ad^0\rhobar)+\sum_{v\in T} (\dim_{\F}H^0(G_v,\ad\rhobar)-\dim_{\F}H^0(G_v,\ad^0\rhobar))\\&+\sum_{v\in S\setminus T}\left(\dim_{\F}
  L(\cD_v)-\dim_{\F} H^0(G_v,\ad^0\rhobar)\right).\end{align*}
From now on for the rest of the notes, we specialise to the case that $F$ is a number field, $S$ is a
finite set of finite places of $F$ including all the places lying over
$p$, and we set $G=G_{F,S}$, $G_v=G_{F_v}$ for $v\in S$. (Since
$G=G_{F,S}$, note in particular that all deformations we are
considering are unramified outside of $S$.) We then
employ standard results on Galois cohomology that can be found
in~\cite{MR2261462}. In particular, we have $H^i(G_{F_v},\ad\rhobar)=0$
if $i\ge 3$, and \[H^i(G_{F,S},\ad^0\rhobar)\cong\oplus_{v\text{
    real}}H^i(G_{F_v},\ad^0\rhobar)=0\] if $i\ge 3$ (the vanishing
of the local cohomology groups follows as $p>2$, so $G_{F_v}$ has
order coprime to that of $\ad^0\rhobar$). Consequently,
$H^i_{\cS,T}(G_{F,S},\ad^0\rhobar)=0$ if $i>3$.

We now employ the local and global Euler characteristic formulas. For
simplicity, assume from now on that $T$ contains all the places of $S$ lying
over $p$. The
global formula gives \[\chi(G_{F,S},\ad^0\rhobar)=
-\sum_{v|\infty}\dim_{\F}H^0(G_{F_v},\ad^0\rhobar)+[F:\Q](n^2-1),\] and
the local formula gives
\begin{align*}\sum_{v\in
  S}\chi(G_{F_v},\ad^0\rhobar)
&=\sum_{v|p}(n^2-1)[F_v:\Q_p]\\
&=(n^2-1)[F:\Q],
  \end{align*}
so that \[\chi_{\cS,T}(G_{F,S},\ad^0\rhobar)=-1+\#T-\sum_{v|\infty}\dim_{\F}H^0(G_{F_v},\ad^0\rhobar)+\sum_{v\in S\setminus T}\left(\dim_{\F}
  L(\cD_v)-\dim_{\F} H^0(G_{F_v},\ad^0\rhobar)\right).\]

Assume now that $\rhobar$ is absolutely irreducible; then
$H^0(G_{F,S},\ad\rhobar)=\F$, so
$H^0_{\cS,T}(G_{F,S},\ad^0\rhobar)=\F$. To say something sensible
about $H^1_{\cS,T}(G_{F,S},\ad^0\rhobar)$ we still need to control the
$H^2_{\cS,T}$ and $H^3_{\cS,T}$. Firstly, the above long exact
sequence gives us in particular the exact sequence \[\begin{tikzpicture}[descr/.style={fill=white,inner sep=1.5pt}]
        \matrix (m) [
            matrix of math nodes,
            row sep=1em,
            column sep=2.5em,
            text height=1.5ex, text depth=0.25ex
        ]
        { & & H^1(G_{F,S},\ad^0\rhobar) &
          \oplus_{v\in T}H^1(G_{F_v},\ad^0\rhobar)\oplus_{v\in S\setminus T}H^1(G_{F_v},\ad^0\rhobar)/L(\cD_v) \\
            & H^2_{\cS,T}(G_{F,S},\ad^0\rhobar) & H^2(G_{F,S},\ad^0\rhobar) &
          \oplus_{v\in S}H^2(G_{F_v},\ad^0\rhobar)\\
                       &H^3_{\cS,T}(G_{F,S},\ad^0\rhobar) &  0.\\
        };

        \path[overlay,->, font=\scriptsize,>=latex]
        (m-1-3) edge (m-1-4)
        (m-1-4) edge[out=355,in=175]  (m-2-2)
        (m-2-2) edge (m-2-3)
        (m-2-3) edge (m-2-4)
        (m-2-4) edge[out=355,in=175]  (m-3-2)
        (m-3-2) edge (m-3-3);
\end{tikzpicture}\]On the other hand, from the Poitou--Tate exact
sequence~\cite[Prop. 4.10, Chapter 1]{MR2261462} we have an exact sequence \[\begin{tikzpicture}[descr/.style={fill=white,inner sep=1.5pt}]
        \matrix (m) [
            matrix of math nodes,
            row sep=1em,
            column sep=2.5em,
            text height=1.5ex, text depth=0.25ex
        ]
        {  H^1(G_{F,S},\ad^0\rhobar) &
          \oplus_{v\in S}H^1(G_{F_v},\ad^0\rhobar)&H^1(G_{F,S},(\ad^0\rhobar)^\vee(1))^\vee&\\
            H^2(G_{F,S},\ad^0\rhobar) &
          \oplus_{v\in
            S}H^2(G_{F_v},\ad^0\rhobar)&H^0(G_{F,S},(\ad^0\rhobar)^\vee(1))^\vee&
          0.\\
                   };

        \path[overlay,->, font=\scriptsize,>=latex]
       (m-1-1) edge (m-1-2)
        (m-1-2) edge (m-1-3)
        (m-1-3) edge[out=355,in=175]  (m-2-1)
        (m-2-1) edge (m-2-2)
        (m-2-2) edge (m-2-3)
        (m-2-3) edge (m-2-4);
\end{tikzpicture}\]Note that $\ad^0\rhobar$ is self-dual under the
trace pairing, so we can and do identify $(\ad^0\rhobar)^\vee(1)$ and
$(\ad^0\rhobar)(1)$. If we let $L(\cD_v)^\perp\subseteq
H^1(G_{F_v},(\ad^0\rhobar)(1))$ denote the annihilator of $L(\cD_v)$
under the pairing coming from Tate local duality, and we
define \[H^1_{\cS,T}(G_{F,S},(\ad^0\rhobar)(1)):=\ker\left(H^1(G_{F,S},(\ad^0\rhobar)(1))\to\oplus_{v\in
    S\setminus
    T}\left(H^1(G_{F_v},(\ad^0\rhobar)(1))/L(\cD_v)^\perp\right)\right),\]then
we deduce that we have an exact sequence \[\begin{tikzpicture}[descr/.style={fill=white,inner sep=1.5pt}]
        \matrix (m) [
            matrix of math nodes,
            row sep=1em,
            column sep=2.5em,
            text height=1.5ex, text depth=0.25ex
        ]
        { & & H^1(G_{F,S},\ad^0\rhobar) &
          \oplus_{v\in T}H^1(G_{F_v},\ad^0\rhobar)\oplus_{v\in S\setminus T}H^1(G_{F_v},\ad^0\rhobar)/L(\cD_v) \\
            & H^1_{\cS,T}(G_{F,S},\ad^0\rhobar(1))^\vee & H^2(G_{F,S},\ad^0\rhobar) &
          \oplus_{v\in S}H^2(G_{F_v},\ad^0\rhobar)\\
                       &H^0(G_{F,S},\ad^0\rhobar(1))^\vee &  0,\\
        };

        \path[overlay,->, font=\scriptsize,>=latex]
        (m-1-3) edge (m-1-4)
        (m-1-4) edge[out=355,in=175]  (m-2-2)
        (m-2-2) edge (m-2-3)
        (m-2-3) edge (m-2-4)
        (m-2-4) edge[out=355,in=175]  (m-3-2)
        (m-3-2) edge (m-3-3);
\end{tikzpicture}\]and comparing with the diagram above shows
that \[H^3_{\cS,T}(G_{F,S},\ad^0\rhobar)\cong
H^0(G_{F,S},\ad^0\rhobar(1))^\vee,\] \[H^2_{\cS,T}(G_{F,S},\ad^0\rhobar)\cong
H^1_{\cS,T}(G_{F,S},\ad^0\rhobar(1))^\vee.\] 

Combining all of this, we see that \begin{align*}\dim_{\F} H^1_{\cS,T}(G_{F,S},\ad^0\rhobar)=&\#T-\sum_{v|\infty}\dim_{\F}H^0(G_{F_v},\ad^0\rhobar)+\sum_{v\in S\setminus T}\left(\dim_{\F}
  L(\cD_v)-\dim_{\F} H^0(G_{F_v},\ad^0\rhobar)\right)\\
&+\dim_{\F}H^1_{\cS,T}(G_{F,S},\ad^0\rhobar(1))-\dim_{\F}
H^0(G_{F,S},\ad^0\rhobar(1)).\end{align*}

Now, similar arguments to those we used above give us the following
result (see Section 2.2 of~\cite{cht}).
\begin{prop}
  \label{prop:H^1 and H^2 of global over local}(1)There is a canonical
  isomorphism
\[\Hom(\m_{R_{\cS}^{\square_T}}/(\m_{R_{\cS}^{\square_T}}^2,\m_{R_{\cS,T}^\loc},\lambda),{\F})\cong
H^1_{\cS,T}(G_{F,S},\ad^0\rhobar).\]

(2) $R_{\cS}^{\square_T}$ is the quotient of a power series ring in
$\dim_{\F}H^1_{\cS,T}(G_{F,S},\ad^0\rhobar)$ variables over
$R_{\cS,T}^\loc$.

(3) The Krull dimension of $R_{\cS}^\univ$ is at least \[1+\sum_{v\in
  S}\left(\textrm{Krull
    dim.}(R^\square_{\rhobar|_{G_{F_v}},\chi}/I(\cD_v))-n^2\right)-\sum_{v|\infty}\dim_{\F}
  H^0(G_{F_v},\ad^0\rhobar)-\dim_{\F} H^0(G_{F,S},\ad^0\rhobar(1)).\]
\end{prop}

\subsection{Finiteness of maps between global deformation rings}
Suppose that $F'/F$ is a finite extension of number fields, and that
$S'$ is the set of places of $F'$ lying over $S$. Assume that
$\rhobar|_{G_{F',S'}}$ is absolutely irreducible. Then restricting the
universal deformation $\rho^\univ$ of $\rhobar$ to $G_{F',S'}$ gives a
ring homomorphism $R_{\rhobar|_{G_{F',S'}}}^\univ\to
  R_{\rhobar}^\univ$. The following very useful fact is due to Khare and Wintenberger.
  \begin{prop}
    \label{prop:finiteness of deformation rings over each other}The
    ring $R_{\rhobar}^\univ$ is a finitely generated $R_{\rhobar|_{G_{F',S'}}}^\univ$-module.
  \end{prop}
  \begin{proof}
    See e.g.\ Lemma 1.2.3 of~\cite{blggt}. 
  \end{proof}
\subsection{Local deformation rings with $l=p$}For proving modularity
lifting theorems, we typically need to consider local deformation
rings when $l=p$ which capture certain properties in $p$-adic Hodge
theory (for example being crystalline with fixed Hodge--Tate weights).
These deformation rings are one of the most difficult and interesting parts
of the subject; for example, a detailed computation of deformation
rings with $l=p=3$ was at the heart of the eventual proof of the
Taniyama--Shimura--Weil conjecture. 

For the most part, the relevant deformation rings 
when $l=p$ are still not well understood; 
we don't have a concrete description of the rings
in most cases, or even basic information such as the number of
irreducible components of the generic fibre. In these notes, we will ignore all of these difficulties, and work
only with the ``Fontaine--Laffaille'' case, where the deformation
rings are formally smooth. This is already enough to have
important applications. 

Assume that $K/\Qp$ is a finite unramified extension, and assume that
$L$ is chosen large enough to contain the images of all embeddings
$K\into\Qpbar$. For each $\sigma:K\into L$, let $H_\sigma$ be a set of
$n$ distinct integers, such that the difference between the maximal
and minimal elements of $H_\sigma$ is less than or equal to $p-2$.

\begin{thm}\label{thm:FL deformation rings}
  There is a unique reduced, $p$-torsion free quotient
  $R^\square_{\rhobar,\chi,\cris,\{H_\sigma\}}$ of
  $R^\square_{\rhobar,\chi}$ with the property that a continuous
  homomorphism $\psi:R^\square_{\rhobar,\chi}\to\Qpbar$ factors
  through $R^\square_{\rhobar,\chi,\cris,\{H_\sigma\}}$ if and only if
  $\psi\circ \rho^\square$ is crystalline, and for each $\sigma:K\into
  L$, we have $\HT_\sigma(\psi\circ \rho^\square)=H_\sigma$. 

Furthermore it has Krull dimension given by $\dim
R^\square_{\rhobar,\chi,\cris,\{H_\sigma\}}=n^2+[K:\Qp]\frac{1}{2}n(n-1)$,
and in fact $R^\square_{\rhobar,\chi,\cris,\{H_\sigma\}}$ is formally
smooth over $\cO$, i.e.\ it is isomorphic to a power series ring in
$n^2-1+[K:\Qp]\frac{1}{2}n(n-1)$ variables over $\cO$.
\end{thm}
In fact, if we remove the assertion of formal smoothness,
Theorem~\ref{thm:FL deformation rings} still holds without the
assumption that $K/\Qp$ is unramified, and without any assumption on
the difference between the maximal and minimal elements of the
$H_\sigma$, but in this case it is a much harder theorem of Kisin
(\cite{MR2373358}). In any case, the formal smoothness will be
important for us. 

Theorem~\ref{thm:FL deformation rings} is essentially a consequence of
Fontaine--Laffaille theory \cite{fl}, which is a form of integral
$p$-adic Hodge theory; it classifies the Galois-stable
lattices in crystalline representations, under the assumptions we've
made above. The first proof of Theorem~\ref{thm:FL deformation rings}
was essentially in Ramakrishna's thesis~\cite{MR1227448}, and the
general result is the content of Section 2.4 of~\cite{cht}. 

\subsection{Local deformation rings with $p\ne l$}In contrast to the
situation when $l=p$, we will need to consider several deformation
problems when $l\ne p$. We will restrict ourselves to the
two-dimensional case. Let $K/\Ql$ be a finite extension, with $l\ne
p$, and fix $n=2$. As we saw in Section~\ref{subsec: monodromy p not
  l}, there is essentially an incompatibility between the wild inertia
subgroup of $G_K$ and the $p$-adic topology on $\GL_2(\cO)$, which
makes it possible to explicitly describe the $p$-adic representations
of $G_K$, and consequently the corresponding universal deformation
rings. This was done in varying degrees of generality over a long
period of time; in particular, in the general $n$-dimensional case we
highlight Section 2.4.4 of~\cite{cht} and~\cite{suhhyun}, and in the
$2$-dimensional setting \cite{pilloninotes}
and~\cite{shotton}. In fact~\cite{shotton} gives a complete
description of the deformation rings for a fixed inertial type.

We will content ourselves with recalling some of the basic structural
results, and with giving a sketch of how the results are proved in one
particular case (see Exercise~\ref{ex: TW deformations} below).

\subsection{Deformations of fixed type}
Recall from
Proposition~\ref{prop: monodromy thm p not l} that given a
representation $\rho:G_K\to\GL_2(\Qpbar)$ there is a Weil--Deligne
representation $\WD(\rho)$ associated to $\rho$. If $\WD=(r,N)$ is a
Weil--Deligne representation, then we write $\WD|_{I_K}$ for
$(r|_{I_K},N)$, and call it an \emph{inertial $\WD$-type}.

Fix $\rhobar:G_K\to\GL_2(\F)$. Then (assuming as usual that $L$ is
sufficiently large) we have the following general
result on $R^\square_{\rhobar,\chi}$ (see e.g.\ Theorem 3.3.1
of~\cite{boecklenotes}).
\begin{thm}
  \label{thm: types are constant on components
    etc}$R^\square_{\rhobar,\chi}$ is equidimensional of Krull dimension $4$, and
  the generic fibre $R^\square_{\rhobar,\chi}[1/p]$ has Krull
  dimension $3$. Furthermore:

(a) The function which takes a $\Qpbar$-point
$x:R^\square_{\rhobar,\chi}[1/p]\to\Qpbar$ to (the isomorphism class
of) $\WD(x\circ\rho^\square)|_{I_K}$ (forgetting $N$) is constant on the irreducible
components of $R^\square_{\rhobar,\chi}[1/p]$.


(b) The irreducible components of $R^\square_{\rhobar,\chi}[1/p]$ are
all regular, and there are only finitely many of them.
\end{thm}
In the light of Theorem~\ref{thm: types are constant on components
    etc}, we make the following definition. Let $\tau$ be an inertial
  $\WD$-type. Then there is a unique reduced, $p$-torsion free
  quotient $R^\square_{\rhobar,\chi,\tau}$ of
  $R^\square_{\rhobar,\chi}$ with the property that a continuous
  homomorphism $\psi:R^\square_{\rhobar,\chi}\to\Qpbar$ factors
  through $R^\square_{\rhobar,\chi,\tau}$ if and only if
  $\psi\circ \rho^\square$ has inertial Weil--Deligne type $\tau$. (Of
  course, for all but finitely many $\tau$, we will just have
  $R^\square_{\rhobar,\chi,\tau}=0$.) By Theorem~\ref{thm: types are constant on components
    etc} we see that if $R^\square_{\rhobar,\chi,\tau}$ is nonzero
  then it has Krull dimension $4$. 
\subsection{Taylor--Wiles deformations}\label{subsec: TW
  deformations}As the name suggests, the deformations that we
consider in this subsection will be of crucial importance for the
Taylor--Wiles--Kisin method. 
Assume that $\rhobar$ is unramified, that $\rhobar(\Frob_K)$ has
distinct eigenvalues, and that $\# k\equiv 1\pmod{p}$. Suppose also that
$\chi$ is unramified.

\begin{lem}\label{universal def ring for TW deformations}
  Suppose that $(\# k-1)$ is exactly divisible by $p^m$. Then 
  $R^\square_{\rhobar,\chi}\cong\cO\llbracket x,y,B,u\rrbracket /((1+u)^{p^m}-1)$. Furthermore,
  if $\varphi\in G_K$ is a lift of $\Frob_K$, then $\rho^\square(\varphi)$
  is conjugate to a diagonal matrix.
\end{lem}
\begin{exercise}\label{ex: TW deformations}Prove this lemma as
  follows. Note firstly that $\rho^\square(P_K)=\{1\}$, because
  $\rhobar(P_K)=\{1\}$, so $\rho^\square(P_K)$ is a pro-$l$-subgroup
  of the pro-$p$-group $\ker(\GL_2(R^\square_{\rhobar,\chi})\to\GL_2(\F))$. 

Let $\varphi$ be a fixed lift of $\Frob_K$ to $G_K/P_K$, and $\sigma$ a
topological generator of $I_K/P_K$, which as in Section~\ref{subsec:
  monodromy p not l} we can choose so that
$\varphi^{-1}\sigma\varphi=\sigma^{\# k}$. Write $\rhobar(\varphi)=
\begin{pmatrix}
  \alphabar&0\\0&\betabar
\end{pmatrix}$, and fix lifts $\alpha,\beta\in\cO$ of $\alphabar,\betabar$.

Then we will show that we can take \[\rho^\square(\varphi)=
  \begin{pmatrix}
    1&y\\x&1
  \end{pmatrix}^{-1}
  \begin{pmatrix}
    \alpha+B&0\\0&\chi(\varphi)/(\alpha+B)
  \end{pmatrix}
  \begin{pmatrix}
    1&y\\x&1
  \end{pmatrix},\] \[\rho^\square(\sigma)=
  \begin{pmatrix}
    1&y\\x&1
  \end{pmatrix}^{-1}
  \begin{pmatrix}
    1+u&0\\0&(1+u)^{-1}
  \end{pmatrix}
  \begin{pmatrix}
    1& y\\x&1
  \end{pmatrix}.\]
  \begin{enumerate}
  \item 
    Let $\rho:G_K\to\GL_2(A)$ be a lift of
    $\rhobar$. By Hensel's lemma, there are $a$, $b\in\m_A$ such that
    $\rho(\varphi)$ has characteristic polynomial
    $(X-(\alpha+a))(X-(\beta+b))$. Show that there are $x$, $y\in\m_A$
    such that \[\rho(\varphi)
    \begin{pmatrix}
      1\\x
    \end{pmatrix}=(\alpha+a)
    \begin{pmatrix}
      1\\x
    \end{pmatrix}\]and  \[\rho(\varphi)
    \begin{pmatrix}
      y\\1
    \end{pmatrix}=(\beta+b)
    \begin{pmatrix}
      y\\1
    \end{pmatrix}\]
  \item Since $\rhobar$ is unramified, $\rhobar(\sigma)=1$, so we
    may write \[
    \begin{pmatrix}
      1&y\\x&1
    \end{pmatrix}^{-1}\rho(\sigma)
    \begin{pmatrix}
      1&y\\x&1
    \end{pmatrix}=
    \begin{pmatrix}
      1+u&v\\w&1+z
    \end{pmatrix}\]with $u$, $v$, $w$, $z\in\m_A$. Use the commutation
    relation between $\rho(\varphi)$ and $\rho(\sigma)$ to show that $v=w=0$.
  \item Use the fact that $\chi$ is unramified to show that $1+z=(1+u)^{-1}$.
  \item Show that $(1+u)^{\#k}=1+u$, and deduce that $(1+u)^{\#k-1}=1$.
  \item Deduce that $(1+u)^{p^m}=1$.
  \item Complete the proof of the lemma.
  \end{enumerate}
\end{exercise}

\subsection{Taylor's ``Ihara avoidance'' deformations}
The following deformation rings are crucial to Taylor's arguments
in~\cite{tay} which avoid the use of Ihara's lemma in proving
automorphy lifting theorems. When $n=2$ these arguments are not
logically necessary, but they are crucial to all applications of
automorphy lifting theorems when $n>2$. They are used in order to
compare Galois representations with differing ramification at places
not dividing~$p$.

Continue to let $K/\Ql$ be a finite extension, and assume that
$\rhobar$ is the trivial $2$-dimensional representation, that
$\#k\equiv 1\pmod{p}$, that $\chi$ is unramified, and that $\chibar$
is trivial. Again, we see that $\rho^\square(P_K)$ is trivial, so that
$\rho^\square$ is determined by the two matrices
$\rho^\square(\sigma)$ and $\rho^\square(\varphi)$, as in
Exercise~\ref{ex: TW deformations}. A similar analysis then yields the
following facts. (For the proof of the analogous results in the
$n$-dimensional case, see Section 3 of~\cite{tay}.)

\begin{defn}
  \leavevmode\begin{enumerate}
  \item Let $\cP_\nr$ be the minimal ideal of
    $R^\square_{\rhobar,\chi}$ modulo which
    $\rho^\square(\sigma)=1_2$. 
  \item For any root of unity $\zeta$ which is trivial modulo~$\lambda$, we let $\cP_\zeta$ be the minimal
    ideal of $R^\square_{\rhobar,\chi}$ modulo which $\rho^\square(\sigma)$ has
    characteristic polynomial $(X-\zeta)(X-\zeta^{-1})$.
  \item Let $\cP_\mult$ be the minimal ideal of $R^\square_{\rhobar,\chi}$ modulo
    which $\rho^\square(\sigma)$ has characteristic polynomial $(X-1)^2$,
    and $\#k(\tr \rho^\square(\varphi))^2=(1+\#k)^2\det \rho^\square(\varphi)$.
  \end{enumerate}
\end{defn}
[The motivation for the definition of $\cP_\mult$ is that we are
attempting to describe the unipotent liftings, and if you assume that $\rho^\square(\sigma)=
\begin{pmatrix}
  1& 1\\0&1
\end{pmatrix}$, this is the relation forced on $\rho^\square(\varphi)$.] 
\begin{prop}
  The minimal primes of $R^\square_{\rhobar,\chi}$ are precisely $\sqrt{\cP_\nr}$,
  $\sqrt{\cP_\mult}$, and the $\sqrt{\cP_\zeta}$ for $\zeta\ne 1$. We
  have $\sqrt{\cP_1}=\sqrt{\cP_\nr}\cap\sqrt{\cP_\mult}$.
\end{prop}
Write $R^\square_{\rhobar,\chi,1}$, $R^\square_{\rhobar,\chi,\zeta}$,
$R^\square_{\rhobar,\chi,\nr}$, $R^\square_{\rhobar,\chi,\mult}$ for the
corresponding quotients of $R^\square_{\rhobar,\chi}$. 
\begin{thm}\label{thm:Ihara avoidance deformations}We have
  $R^\square_{\rhobar,\chi,1}/\lambda=R^\square_{\rhobar,\chi,\zeta}/\lambda$. Furthermore,
  \begin{enumerate}
  \item If  $\zeta\ne 1$ then $R^\square_{\rhobar,\chi,\zeta}[1/p]$ is 
    geometrically irreducible of dimension $3$.
  \item  $R^\square_{\rhobar,\chi,\nr}$ is formally smooth over $\cO$ (and thus
    geometrically irreducible) of relative dimension $3$.
  \item  $R^\square_{\rhobar,\chi,\mult}[1/p]$ is 
    geometrically irreducible of dimension $3$.
  \item $\Spec R^\square_{\rhobar,\chi,1}=\Spec
    R^\square_{\rhobar,\chi,\nr}\cup\Spec
    R^\square_{\rhobar,\chi,\mult}$ and 
    $\Spec R^\square_{\rhobar,\chi,1}/\lambda=\Spec R^\square_{\rhobar,\chi,\nr}/\lambda\cup\Spec
    R^\square_{\rhobar,\chi,\mult}/\lambda$ are both a union of two
    irreducible components, and have relative dimension $3$.
  \end{enumerate}
\end{thm} 
\begin{proof}
  See Proposition 3.1 of~\cite{tay} for an $n$-dimensional version of
  this result. In the 2-dimensional case it can be proved by
  explicitly computing equations for the lifting rings; see~\cite{shotton}.
\end{proof}

\section{Modular and automorphic forms, and the Langlands
  correspondence}We now turn to the automorphic side of the Langlands
correspondence, and define the spaces of modular forms to which our
modularity lifting theorems pertain.
\subsection{The local Langlands correspondence (and the
  Jacquet--Langlands correspondence)}\label{subsec:loclanglands}

Weil--Deligne representations are the objects on the ``Galois'' side of
the local Langlands correspondence. We now describe the objects on the
``automorphic'' side. These will be representations $(\pi,V)$ of
$\GL_n(K)$ on (usually infinite-dimensional) $\C$-vector spaces, where
as above $K/\Ql$ is a finite extension for some prime~$l$.
\begin{defn}We say that $(\pi,V)$ is \emph{smooth} if for any vector
  $v\in V$, the stabiliser of $v$ in $\GL_n(K)$ is open. We say that
  $(\pi,V)$ is \emph{admissible} if it is smooth, and for any compact open
  subgroup $U\subset \GL_n(K)$, $V^U$ is finite-dimensional.
  
\end{defn}For example, a smooth one-dimensional representation of $K^\times$ is the
same thing as a continuous character (for the discrete topology on~$\C$).
\begin{fact}
  \leavevmode\begin{enumerate}
  \item If $\pi$ is smooth and irreducible then it is admissible.
  \item Schur's lemma holds for admissible smooth representations, and in particular if $\pi$ is smooth,
    admissible and irreducible then it has a central character $\chi_\pi:K^\times\to\C^\times$.
  \end{enumerate}
\end{fact}
In general these representations are classified in terms of the
(super)cuspidal representations. We won't need the details of this
classification, and accordingly we won't define the cuspidal
representations (for which see for example Chapter IV of~\cite{MR2234120}).

Let $B$ be the subgroup of $\GL_2(K)$ consisting of upper-triangular
matrices. Define $\delta:B\to K^\times$ by \[\delta\left(
\begin{pmatrix}
  a&*\\0& d
\end{pmatrix}\right)=ad^{-1}.\]Given two continuous characters $\chi_1$,
$\chi_2:K^\times\to\C^\times$, we may view $\chi_1\otimes\chi_2$ as a
representation of $B$ by \[\chi_1\otimes\chi_2:
\begin{pmatrix}
  a&*\\0&d
\end{pmatrix}\mapsto\chi_1(a)\chi_2(d).\]Then we define a
representation $\chi_1\times\chi_2$ of $\GL_2(K)$ by \emph{normalized induction}:
\begin{align*}
  \chi_1\times\chi_2&=\nInd_B^{\GL_2(K)}(\chi_1\otimes\chi_2)\\
  &:=\{\varphi:GL_2(K)\to
  \C|\varphi(hg)=(\chi_1\otimes\chi_2)(h)|\delta(h)|_K^{1/2}\varphi(g)\text{
    for all }h\in B,\ g\in\GL_2(K)\}
\end{align*}where $\GL_2(K)$ acts by $(g\varphi)(g')=\varphi(g'g)$, and we
only allow smooth $\varphi$, i.e.\ functions for which there is an open
subgroup $U$ of $\GL_2(K)$ such that $\varphi(gu)=\varphi(g)$ for all
$g\in\GL_2(K)$, $u\in U$.

The representation $\chi_1\times \chi_2$ has length at most 2, but is
not always irreducible. It is always the case that $\chi_1\times
\chi_2$ and $\chi_2\times\chi_1$ have the same Jordan-H\"older
factors. If $\chi_1\times\chi_2$ is irreducible then we say that it is
a \emph{principal series} representation.

\begin{fact}
  \leavevmode\begin{enumerate}
  \item $\chi_1\times\chi_2$ is irreducible unless
    $\chi_1/\chi_2=|\cdot|_K^{\pm 1}$.
  \item $\chi\times \chi|\cdot|_K$ has a one-dimensional irreducible
    subrepresentation, and the corresponding quotient is
    irreducible. We denote this quotient by $\Sp_2(\chi)$.
  \end{enumerate}
\end{fact}
We will let $\chi_1\boxplus\chi_2$ denote $\chi_1\times\chi_2$ unless
$\chi_1/\chi_2=|\cdot|_K^{\pm 1}$, and we let
\[\chi\boxplus\chi|\cdot|_K=\chi|\cdot|_K\boxplus\chi=(\chi |\cdot |_K^{1/2}) \circ
\det.\] (While this notation may seem excessive, we remark that a similar
construction is possible for $n$-dimensional representations, which is where the
notation comes from.) These representations, and the $\Sp_2(\chi)$, are all the
non-cuspidal irreducible admissible representations of $\GL_2(K)$. We say that
an irreducible smooth representation $\pi$ of $\GL_2(K)$ is \emph{discrete
  series} if it is of the form $\Sp_2(\chi)$ or is cuspidal.

The local Langlands correspondence provides a unique family of
bijections $\rec_K$ from the set of isomorphism classes of irreducible smooth representations
of $\GL_n(K)$ to the set of isomorphism classes of $n$-dimensional Frobenius semisimple
Weil--Deligne representations of $W_K$ over $\C$, satisfying a list of
properties. In order to be uniquely determined, one needs to formulate
the correspondence for all $n$ at once, and the properties are
expressed in terms of $L$- and $\varepsilon$-factors, neither of which we
have defined. Accordingly, we will not make a complete statement of
the local Langlands correspondence, but will rather state the
properties of the correspondence that we will need to use. (Again, the
reader could look at the book~\cite{MR2234120} for these properties,
and many others.) It is also
possible to define the correspondence in global terms, as we will see
later, and indeed at present the only proof of the correspondence is
global. 
\begin{fact}
  We now list some properties of $\rec_K$ for $n=1$, $2$.
  \begin{enumerate}
  \item If $n=1$ then $\rec_K(\pi)=\pi\circ\Art_K^{-1}$.
  \item If $\chi$ is a smooth character, $\rec_K(\pi\otimes(\chi\circ\det))=\rec_K(\pi)\otimes\rec_K(\chi)$.
  \item $\rec_K(\Sp_2(\chi))=\Sp_2(\rec_K(\chi))$. (See Exercise~\ref{ex: WD reps} for
    this notation.)
  \item $\rec_K(\chi_1\boxplus\chi_2)=\rec_K(\chi_1)\oplus\rec_K(\chi_2)$.
  \item if~$n=2$, then $\rec_K(\pi)$ is unramified (i.e.\ $N=0$ and the restriction to
    $I_K$ is trivial) if and only if $\pi=\chi_1\boxplus\chi_2$ with
    $\chi_1$, $\chi_2$ both unramified characters
    (i.e.\ trivial on $\cO_K^\times$). These
    conditions are equivalent to $\pi^{\GL_2(\cO_K)}\ne 0$, in which
    case it is one-dimensional.
  \item $\pi$ is discrete series if and only if $\rec_K(\pi)$ is
    indecomposable, and cuspidal if and only if $\rec_K(\pi)$ is
    irreducible.
  \end{enumerate}

\end{fact}
\subsection{Hecke operators}Consider the set of compactly supported $\C$-valued
functions on $\GL_2(\cO_K)\backslash\GL_2(K)/\GL_2(\cO_K)$. Concretely,
these are functions which vanish outside of a finite number of double
cosets $\GL_2(\cO_K)g\GL_2(\cO_K)$. The set of
such functions is in fact a ring, with the multiplication being given
by convolution. To be precise, we fix $\mu$ the (left and right) Haar
measure on $\GL_2(K)$ such that $\mu(\GL_2(\cO_K))=1$, and we
define \[(\varphi_1*\varphi_2)(x)=\int_{\GL_2(K)}\varphi_1(g)\varphi_2(g^{-1}x)d\mu_g.\]
Of course, this integral is really just a finite sum. One can check
without too much difficulty that the ring $\cH$ of these Hecke operators is
just $\C[T,S^{\pm 1}]$, where $T$ is the characteristic function
of \[\GL_2(\cO_K)
\begin{pmatrix}
  \varpi_K&0\\0&1
\end{pmatrix}\GL_2(\cO_K)\]and $S$ is the characteristic function
of \[\GL_2(\cO_K)
\begin{pmatrix}
  \varpi_K&0\\0&\varpi_K
\end{pmatrix}\GL_2(\cO_K).\]

The algebra $\cH$ acts on an irreducible admissible
$\GL_2(K)$-representation $\pi$. Given $\varphi\in\cH$, we obtain a
linear map $\pi(\varphi):\pi\to\pi^{\GL_2(\cO_K)}$,
by \[\pi(\varphi)(v)=\int_{\GL_2(K)}\varphi(g)\pi(g)vd\mu_g.\] In
particular, if $\pi$ is unramified then $\pi(\varphi)$ acts via a scalar
on the one-dimensional $\C$-vector space $\pi^{\GL_2(\cO_K)}$. We will
now compute this scalar explicitly.

\begin{exercise}\label{ex:computing Hecke evals on unram PS}
  \leavevmode\begin{enumerate}
  \item Show that we have decompositions \[\GL_2(\cO_K)
\begin{pmatrix}
  \varpi_K&0\\0&\varpi_K
\end{pmatrix}\GL_2(\cO_K)=\begin{pmatrix}
  \varpi_K&0\\0&\varpi_K
\end{pmatrix}\GL_2(\cO_K),\]and  \[\GL_2(\cO_K)
\begin{pmatrix}
  \varpi_K&0\\0&1
\end{pmatrix}\GL_2(\cO_K)=\left(\coprod_{\alpha\in\cO_K\pmod{\varpi_K}}
\begin{pmatrix}
  \varpi_K&\alpha\\0&1
\end{pmatrix}\GL_2(\cO_K)\right)\coprod
\begin{pmatrix}
  1&0\\0&\varpi_K
\end{pmatrix}\GL_2(\cO_K).\]
\item Suppose that $\pi=(\chi|\cdot|^{1/2})\circ\det$ with $\chi$
  unramified. Show that $\pi^{\GL_2(\cO_K)}=\pi$, and that $S$ acts
  via $\chi(\varpi_K)^2(\# k)^{-1}$, and that $T$ acts via $(\#k^{1/2}+\#k^{-1/2})\chi(\varpi_K)$.
\item Suppose that $\chi_1$, $\chi_2$ are unramified characters and
  that $\chi_1\ne\chi_2|\cdot |_K^{\pm 1}$. Let
  $\pi=\chi_1\boxplus\chi_2$. Using the Iwasawa decomposition
  $\GL_2(K)=B(K)\GL_2(\cO_K)$, check that $\pi^{\GL_2(\cO_K)}$ is
  one-dimensional, and is spanned by a function $\varphi_0$ with $\varphi_0\left(
  \begin{pmatrix}
    a& b\\0&d
  \end{pmatrix}\right)=\chi_1(a)\chi_2(d)|a/d|^{1/2}$. Show that $S$ acts on
  $\pi^{\GL_2(\cO_K)}$ via $(\chi_1\chi_2)(\varpi_K)$, and that
  $T$ acts via $\#k^{1/2}(\chi_1(\varpi_K)+\chi_2(\varpi_K))$.

\end{enumerate}
\end{exercise}

\subsection{Modular forms and automorphic forms on quaternion
  algebras}\label{subsec: modular forms}
Let $F$ be a totally real field, and let $D/F$ be a quaternion algebra
with centre $F$, i.e.\ a central simple
$F$-algebra of dimension 4. Letting $S(D)$ be the set of places $v$ of $F$
at which $D$ is ramified, i.e.\ for which $D\otimes_FF_v$ is a division
algebra (equivalently, is not isomorphic to $M_2(F_v)$), it is known
that $S(D)$ classifies $D$ up to isomorphism, and that $S(D)$ can be
any finite set of places of $F$ of even cardinality (so for example
$S(D)$ is empty if and only if $D=M_2(F)$). We will now define some spaces of
automorphic forms on $D^\times$.

For each $v|\infty$ fix $k_v\ge 2$ and $\eta_v\in\Z$ such that
$k_v+2\eta_v-1=w$ is independent of $v$. These will be the weights of
our modular forms. Let $G_D$ be the algebraic group over $\Q$ such
that for any $\Q$-algebra $R$, $G_D(R)=(D\otimes_\Q R)^\times$. For
each place $v|\infty$ of $F$, we define a subgroup $U_v$ of
$(D\otimes_F F_v)^\times$ as follows: if $v\in S(D)$ we let
$U_v=(D\otimes_F F_v)^\times\cong\mathbb{H}^\times$ (where
$\mathbb{H}$ denotes the Hamilton quaternions), and if $v\notin
S(D)$, so that $(D\otimes_F F_v)^\times\cong\GL_2(\R)$, we take
$U_v=\R^\times\SO(2)$. If $\gamma=
\begin{pmatrix}
  a&b\\c&d
\end{pmatrix}\in\GL_2(\R)$ and $z\in\C-\R$, we let
$j(\gamma,z)=cz+d$. One checks easily that
$j(\gamma\delta,z)=j(\gamma,\delta z)j(\delta,z)$. 

We now define a representation $(\tau_v,W_v)$ of $U_v$ over $\C$ for
each $v|\infty$. If $v\in S(D)$, we have
$U_v\into\GL_2(\overline{F}_v)\cong\GL_2(\C)$ which acts on $\C^2$, and
we let $(\tau_v,W_v)$ be the
representation \[(\Sym^{k_v-2}\C^2)\otimes(\Wedge^2\C^2)^{\eta_v}.\]
If $v\notin S(D)$, then we have $U_v\cong \R^\times \SO(2)$, and we
take $W_v=\C$,
with \[\tau_v(\gamma)=j(\gamma,i)^{k_v}(\det\gamma)^{\eta_v-1}.\] We
write $U_\infty=\prod_{v|\infty}U_v$,
$W_\infty=\otimes_{v|\infty}W_v$,
$\tau_\infty=\otimes_{v|\infty}\tau_v$. Let $\A=\A_\Q$ be the adeles
of $\Q$, and let $\A^\infty$ be the finite adeles. We then define
$S_{D,k,\eta}$ (where $k,\eta$ reflect the dependence on the integers $k_v,\eta_v$)
to be the space of functions $\varphi:G_D(\Q)\backslash G_D(\A)\to
W_\infty$ which satisfy
\begin{enumerate}
\item $\varphi(gu_\infty)=\tau_\infty(u_\infty)^{-1}\varphi(g)$ for all
  $u_\infty\in U_\infty$ and $g\in G_D(\A)$.
\item There is a nonempty open subset $U^\infty\subset G_D(\A^\infty)$
  such that $\varphi(gu)=\varphi(g)$ for all $u\in U^\infty$, $g\in G_D(\A)$.
\item Let $S_\infty$ denote the infinite places of $F$. If $g\in
  G_D(\A^\infty)$ then the function \[(\C-\R)^{S_\infty-S(D)}\to
  W_\infty\] defined by
 \[h_\infty(i,\dots,i)\mapsto \tau_\infty(h_\infty)\phi(gh_\infty)\]
  is holomorphic. [Note that this
  function is well-defined by the first condition, as  $U_\infty$ is
  the stabiliser of
  $(i,\dots,i)$.]
\item If $S(D)=\emptyset$ then for all $g\in G_D(\A)=\GL_2(\A_F)$,
  we have  \[\int_{F\backslash\A_F}\varphi(
  \begin{pmatrix}
    1& x\\0&1
  \end{pmatrix}g)dx=0.\] If in addition we have $F=\Q$, then we
  furthermore demand that for all $g\in G_D(\A^\infty)$,
  $h_\infty\in\GL_2(\R)^+$ the function $\varphi(gh_\infty)|\Im (h_\infty i)|^{k/2}$
   is bounded on $\C-\R$.

\end{enumerate}
There is a natural action of $G_D(\A^\infty)$ on $S_{D,k,\eta}$ by
right-translation, i.e.\ $(g\varphi)(x):=\varphi(xg)$.
\begin{exercise}
  While this definition may at first sight appear rather mysterious,
  it is just a generalisation of the familiar spaces of cuspidal
  modular forms. For example, take $F=\Q$, $S(D)=\emptyset$,
  $k_\infty=k$, and
  $\eta_\infty=0$. Define \[U_1(N)=\{g\in \GL_2(\Zhat)|g\equiv
  \begin{pmatrix}
    *&*\\0&1
  \end{pmatrix}\pmod{N}\}.\]
  \begin{enumerate}
  \item Let $\GL_2(\Q)^{+}$ be the subgroup of $\GL_2(\Q)$ consisting
    of matrices with positive determinant. Show that the intersection
    of $\GL_2(\Q)^+$ and $U_1(N)$ inside $\GL_2(\A^\infty)$ is
    $\Gamma_1(N)$,  the matrices in $\SL_2(\Z)$ congruent to $
    \begin{pmatrix}
      1&*\\0&1
    \end{pmatrix}\pmod{N}$. [Hint: what is $\Zhat^\times\cap\Q^\times$?]
  \item Use the facts that $\GL_2(\A)=\GL_2(\Q)U_1(N)\GL_2(\R)^+$
    [which follows from strong approximation for $\SL_2$ and the fact
    that $\det U_1(N)=\Zhat^\times$] and that
    $\A^\times=\Q^\times\Zhat^\times\R^\times_{>0}$ to show that
    $S_{D,k,0}^{U_1(N)}$ can naturally be identified with a space of
    functions \[\varphi:\Gamma_1(N)\backslash \GL_2(\R)^+\to\C\]
    satisfying \[\varphi(gu_\infty)=j(u_\infty,i)^{-k}\varphi(g)\]for all
    $g\in\GL_2(\R)^+$, $u_\infty\in\R^\times_{>0} \SO(2)$.
  \item Show that the stabiliser of $i$ in $\GL_2(\R)^+$ is
    $\R^\times_{>0} \SO(2)$. Hence deduce a natural isomorphism between
    $S_{D,k,0}^{U_1(N)}$ and $S_k(\Gamma_1(N))$, which takes a
    function $\varphi$ as above to the function $(gi\mapsto
    j(g,i)^k\varphi(g))$, $g\in\GL_2(\R)^+$.

  \end{enumerate}

\end{exercise}
The case that $S_\infty\subset S(D)$ is particularly simple; then if
$U\subset G_D(\A^\infty)$ is an open subgroup, then $S_{D,2,0}^U$ is just the set of
$\C$-valued functions on \[G_D(\Q)\backslash G_D(\A)/G_D(\R)U,\] which
is a finite set. When proving modularity lifting theorems, we will be
able to reduce to the case that $S_\infty\subset S(D)$; when this
condition holds, we say that $D$ is a \emph{definite} quaternion algebra.

We will now examine the action of Hecke operators on these
spaces. Choose an $\cO_F$-order $\cO_D\subset D$ (that is, an $\cO_F$-subalgebra
of $D$ which is finitely generated as a $\Z$-module and for which
$\cO_D\otimes_{\cO_F}F\isoto D$). For example, if $D=M_2(F)$, one may
take $\cO_D=M_2(\cO_F)$.

 For all but finite many finite places
$v$ of $F$ we can choose an isomorphism $D_v\cong M_2(F_v)$ such that
this isomorphism induces an isomorphism
$\cO_D\otimes_{\cO_F}\cO_{F_v}\isoto M_2(\cO_{F_v})$. Then
$G_D(\A^\infty)$ is the subset of elements
$g=(g_v)\in\prod_{v\nmid\infty}G_D(F_v)$ such that
$g_v\in\GL_2(\cO_{F_v})$ for almost all $v$.

We now wish to describe certain irreducible representations of
$G_D(\A^\infty)$ in terms of irreducible representations of the
$GL_2(F_v)$. More generally, we have the following construction. Let
$I$ be an indexing set and for each $i\in I$, let $V_i$ be a $\C$-vector space. Suppose
that we are given $0\ne e_i\in V_i$ for almost all $i$ (that is, all
but finitely many $i$). Then we
define the \emph{restricted tensor product} \[\otimes'_{\{e_i\}} V_i:=\varinjlim_{J\subseteq I}\otimes_{i\in
  J}V_i,\]where the colimit is over the finite subsets $J\subseteq I$ containing all the places for which $e_i$ is
not defined, and where the transition maps for the colimit are given
by ``tensoring with the $e_i$''. It can be checked that
$\otimes'_{\{e_i\}}V_i\cong\otimes'_{\{f_i\}}V_i$ if for almost
all $i$, $e_i$ and $f_i$ span the same line.
\begin{defn}
  We call a representation $(\pi,V)$ of $G_D(\A^\infty)$
  \emph{admissible} if
  \begin{enumerate}
  \item for any $x\in V$, the stabiliser of $x$ is open, and
  \item for any $U\subset G_D(\A^\infty)$ an open subgroup, $\dim_\C V^U<\infty$.
  \end{enumerate}

\end{defn}
\begin{fact}(See~\cite{MR546596}.)
  If $\pi_v$ is an irreducible smooth (so admissible) representation
  of $(D\otimes_F F_v)^\times$ with $\pi_v^{\GL_2(\cO_{F_v})}\ne 0$ for almost all
  $v$, then
  $\otimes'\pi_v:=\otimes'_{\left\{\pi_v^{\GL_2(\cO_{F_v})}\right\}}\pi_v$ is an
  irreducible admissible smooth representation of $G_D(\A^\infty)$, and any
  irreducible admissible smooth representation of $G_D(\A^\infty)$ arises in this way
  for unique $\pi_v$.
\end{fact}
We have a global Hecke algebra, which decomposes as a restricted tensor
product of the local Hecke algebras in the following way. For each
finite place $v$ of $F$ we choose $U_v\subset (D\otimes_FF_v)^\times$ a compact open
subgroup, such that $U_v=\GL_2(\cO_{F_v})$ for almost all $v$. Let
$\mu_v$ be a Haar measure on $(D\otimes_FF_v)^\times$, chosen such that for almost
all $v$ we have $\mu_v(\GL_2(\cO_{F_v}))=1$. Then there is a unique
Haar measure $\mu$ on $G_D(\A^\infty)$ such that for any $U_v$ as
above, if we set $U=\prod_vU_v\subset G_D(\A^\infty)$, then
$\mu(U)=\prod_v\mu_v(U_v)$. Then there is a
decomposition \[\cC_c(U\backslash
G_D(\A^\infty)/U)\mu\cong\otimes'_{\left\{1_{U_v}\mu_v\right\}}\cC_c(U_v\backslash
(D\otimes_FF_v)^\times/U_v)\mu_v,\]and the actions of these Hecke algebras are
compatible with the decomposition $\pi=\otimes'\pi_v$. For the
following fact, see Lemma~1.3 of~\cite{tay-fm2}.
\begin{fact}
  $S_{D,k,\eta}$ is a semisimple admissible representation of $G_D(\A^\infty)$.
\end{fact}
\begin{defn}
  The irreducible constituents of $S_{D,k,\eta}$ are called the
  \emph{cuspidal automorphic representations} of $G_D(\A^\infty)$ of
  weight $(k,\eta)$.
\end{defn}
\begin{remark}
  Note that these automorphic representations do not include Maass
  forms or weight one modular forms; they are the class of
  \emph{regular algebraic} or \emph{cohomological} cuspidal automorphic representations.
\end{remark}For the following facts, the reader could consult~\cite{MR0379375}.
\begin{fact}
  (Strong multiplicity one (and multiplicity one) for $\GL_2$) Suppose that
  $S(D)=\emptyset$. Then every irreducible constituent of
  $S_{D,k,\eta}$ has multiplicity one. In fact if $\pi$ (respectively
  $\pi'$) is a cuspidal automorphic representation of weight
  $(k,\eta)$ (respectively $(k',\eta')$) such that $\pi_v\cong\pi'_v$
  for almost all $v$ then $k=k'$, $\eta=\eta'$, and $\pi=\pi'$.
\end{fact}
\begin{fact}
  (The theory of newforms) Suppose that
  $S(D)=\emptyset$. If $\n$ is an ideal of $\cO_F$, write \[U_1(\n)=\{g\in\GL_2(\hat{\cO}_F)|g\equiv
  \begin{pmatrix}
    *&*\\0&1
  \end{pmatrix}\pmod{\n}\}.\] If $\pi$ is a cuspidal automorphic
  representation of $G_D(\A^\infty)$ then there is a unique ideal
  $\n$ such that $\pi^{U_1(\n)}$ is one-dimensional, and
  $\pi^{U_1(\m)}\ne 0$ if and only if $\n|\m$. We call $\n$ the
  \emph{conductor} (or sometimes the \emph{level}) of $\pi$.

\end{fact}

Analogous to the theory of admissible representations of $\GL_2(K)$,
$K/\Qp$ finite that we sketched above, there is a theory of admissible
representations of $M^\times$, $M$ a nonsplit quaternion algebra over
$K$. Since $M^\times/K^\times$ is compact, any irreducible smooth
representation of $M^\times$ is finite-dimensional. There is a
bijection $\JL$, the \emph{local Jacquet--Langlands correspondence},
from the irreducible smooth representations of $M^\times$ to the
discrete series representations of $\GL_2(K)$, determined by a
character identity.

\begin{fact}[The global Jacquet--Langlands correspondence]We have the
  following facts about $G_D(\A^\infty)$.

  \begin{enumerate}
  \item The only finite-dimensional cuspidal automorphic
    representations of $G_D(\A^\infty)$
    are 1-dimensional representations which factor through the
    reduced norm; these only exist if $D\ne M_2(F)$.
  \item     There is a bijection $\JL$ from the infinite-dimensional cuspidal
    automorphic representations of $G_D(\A^\infty)$ of weight
    $(k,\eta)$ to the cuspidal automorphic representations of
    $\GL_2(\A_F^\infty)$ of weight $(k,\eta)$ which are
    discrete series for all finite places $v\in S(D)$. Furthermore if
    $v\notin S(D)$ then $\JL(\pi)_v=\pi_v$, and if $v\in S(D)$ then
    $\JL(\pi)_v=\JL(\pi_v)$.
  \end{enumerate}

\end{fact}
\begin{rem}
  We will use the global Jacquet--Langlands correspondence together
  with base change (see below) to reduce ourselves to considering the
  case that $S(D)=S_\infty$ when proving automorphy lifting theorems.
\end{rem}
\subsection{Galois representations associated to automorphic representations}

\begin{fact}[The existence of Galois representations associated to
  regular algebraic cuspidal
  automorphic representations] Let $\pi$ be a regular algebraic cuspidal automorphic
  representation of $\GL_2(\A_F^\infty)$ of weight $(k,\eta)$. Then
  there is a CM field $L_\pi$ which contains the eigenvalues of
  $T_v$ and $S_v$ on $\pi_v^{\GL_2(\cO_{F_v})}$ for each finite
  place~$v$ at which~$\pi_v$ is unramified. 
  Furthermore, for each finite place $\lambda$ of $L_\pi$ there is a continuous irreducible Galois
  representation \[r_\lambda(\pi):G_F\to\GL_2(\Lbar_{\pi,\lambda})\]
  such that
  \begin{enumerate}
  \item if $\pi_v$ is unramified and $v$ does not divide the residue
    characteristic of $\lambda$, then $r_\lambda(\pi)|_{G_{F_v}}$
  is unramified, and the characteristic polynomial of $\Frob_v$ is
  $X^2-t_vX+(\#k(v))s_v$, where $t_v$ and $s_v$ are the eigenvalues of
  $T_v$ and $S_v$ respectively on $\pi_v^{\GL_2(\cO_{F_v})}$, and
  $k(v)$ is the residue field of~$F_v$. [Note
  that by the Chebotarev density theorem, this already characterises
  $r_\lambda(\pi)$ up to isomorphism.]
\item More generally, for all finite places~
  $v$ not dividing the residue characteristic of~ $\lambda$, $\WD(r_\lambda(\pi)|_{G_{F_v}})^{F-\operatorname{ss}}\cong\rec_{F_v}(\pi_v\otimes|\det|^{-1/2})$.
\item If $v$ divides the residue characteristic of $\lambda$ then
  $r_\lambda(\pi)|_{G_{F_v}}$ is de Rham with $\tau$-Hodge-Tate weights
  $\eta_\tau,\eta_\tau+k_\tau-1$, where $\tau:F\into
  \barL_{\pi}\subset \C$ is an embedding lying over $v$. If $\pi_v$ is
  unramified then $r_\lambda(\pi)|_{G_{F_v}}$ is crystalline.
\item If $c_v$ is a complex conjugation, then $\det r_\lambda(\pi)(c_v)=-1$.
  \end{enumerate}
  \begin{rem}
    The representations~$r_\lambda(\pi)$ in fact form a strictly
    compatible system; see Section 5 of \cite{blggt} for a discussion
    of this in a more general context.
  \end{rem}
  \begin{rem}Using the Jacquet--Langlands correspondence, we get
    Galois representations for the infinite-dimensional cuspidal
    automorphic representations of $G_D(\A^\infty)$ for any $D$. In
    fact, the proof actually uses the Jacquet--Langlands
    correspondence; in most cases, you can transfer to a $D$ for which
    $S(D)$ contains all but one infinite place, and the Galois
    representations are then realised in the \'etale cohomology of the
    associated Shimura curve. The remaining Galois representations are
    constructed from these ones via congruences.
  \end{rem}
\end{fact}
\begin{fact}[Cyclic base change]\label{fact: cyclic base change}
  Let $E/F$ be a cyclic extension of totally real fields of prime
  degree. Let $\Gal(E/F)=\langle\sigma\rangle$ and let 
  $\Gal(E/F)^\vee=\langle\delta_{E/F}\rangle$ (here $\Gal(E/F)^\vee$
  is the dual abelian group of $\Gal(E/F)$). Let $\pi$ be a cuspidal
  automorphic representation of $\GL_2(\A^\infty_F)$ of weight
  $(k,\eta)$. Then there is a cuspidal automorphic representation
  $\BC_{E/F}(\pi)$ of $\GL_2(\A_E^\infty)$ of weight
  $(\BC_{E/F}(k),\BC_{E/F}(\eta))$ such that
  \begin{enumerate}
  \item for all finite places $v$ of $E$,
    $\rec_{E_v}(\BC_{E/F}(\pi)_v)=(\rec_{F_{v|F}}(\pi_{v|_F}))|_{W_{E_v}}$. In
    particular, $r_\lambda(BC_{E/F}(\pi))\cong r_\lambda(\pi)|_{G_E}$.
  \item $\BC_{E/F}(k)_v=k_{v|_F}$, $\BC_{E/F}(\eta)_v=\eta_{v|_F}$.
  \item $\BC_{E/F}(\pi)\cong\BC_{E/F}(\pi')$ if and only if
    $\pi\cong\pi'\otimes(\delta_{E/F}^i\circ\Art_F\circ\det)$ for some $i$.
  \item A cuspidal automorphic representation $\pi$ of
    $\GL_2(\A^\infty_E)$ is in the image of $\BC_{E/F}$ if and only if
    $\pi\circ\sigma\cong\pi$. 
  \end{enumerate}

\end{fact}
\begin{defn}
  We say that $r:G_F\to\GL_2(\Qpbar)$ is \emph{modular} (of weight
  $(k,\eta)$) if it is isomorphic to $r_\lambda(\pi)$ for some
  cuspidal automorphic representation $\pi$ (of weight $(k,\eta)$) and
  some place $\lambda$ of $L_{\pi}$ lying over~$p$.
\end{defn}
\begin{prop}
  Suppose that $r:G_F\to\GL_2(\Qpbar)$ is a continuous representation,
  and that $E/F$ is a finite solvable Galois extension of totally real
  fields. Then $r|_{G_E}$ is modular if and only if $r$ is modular.
\end{prop}
\begin{exercise}
  Prove the above proposition as follows.
  \begin{enumerate}
  \item Use induction to reduce to the case that $E/F$ is cyclic of
    prime degree.
  \item Suppose that $r|_{G_E}$ is modular, say $r|_{G_E}\cong
    r_\lambda(\pi)$. Use strong multiplicity one to show that
    $\pi\circ\sigma\cong\pi$. Deduce that there is an automorphic
    representation $\pi'$ such that $\BC_{E/F}(\pi')=\pi$.
  \item Use Schur's lemma to deduce that there is a character $\chi$
    of $G_F$ such that $r\cong r_\lambda(\pi')\otimes
    \chi$. Conclude that $r$ is modular.
  \end{enumerate}

\end{exercise}
We can make use of this result to make considerable simplifications in
our proofs of modularity lifting theorems. It is frequently employed
in conjunction with the following fact from class field theory.
\begin{fact}[Lemma 2.2 of~\cite{MR1981033}]\label{fact: existence of number fields with given local properties}Let $K$ be a number field, and let $S$ be a finite set of
  places of $K$. For each $v\in S$, let $L_v$ be a finite Galois
  extension of $K_v$. Then there is a finite solvable Galois extension
  $M/K$ such that for each place $w$ of $M$ above a place $v\in S$
  there is an isomorphism $L_v\cong M_w$ of $K_v$-algebras.
  
\end{fact}
Note that we are allowed to have infinite places in $S$, so that if
$K$ is totally real we may choose to make $L$ totally real by an
appropriate choice of the $L_v$.

\section{The Taylor--Wiles--Kisin method}
In this section we prove our modularity lifting theorem, using the
Taylor--Wiles--Kisin patching method. Very roughly, the idea of this
method is to patch together spaces of modular forms of varying levels,
allowing more and more ramification at places away from~$p$, in such a
way as to ``smooth out'' the singularities of global deformation
rings, reducing the problem to one about local deformation rings. This
patching procedure is (at least on first acquaintance) somewhat
strange, as it involves making many non-canonical choices to identify
spaces of modular forms with level structures at different primes.
\subsection{}
Our aim now is to prove the
following theorem. Let $p>3$ be a prime, and let $L/\Qp$ be a finite
extension with ring of integers $\cO$, maximal ideal $\lambda$, and
residue field $\F=\cO/\lambda$. Let $F$ be a totally real number
field, and assume that $L$ is sufficiently large that $L$ contains the
images of all embeddings $F\into\Lbar$. 
\begin{thm}
  \label{thm: main modularity lifting theorem}Let
  $\rho,\rho_0:G_F\to\GL_2(\cO)$ be two continuous representations,
  such that $\rhobar=\rho\pmod{\lambda}=\rho_0\pmod{\lambda}$. Assume
  that $\rho_0$ is modular, that $\rho$ is geometric, and that $p>3$.
  Assume further that the following properties hold.
  \begin{enumerate}
  \item For all $\sigma:F\into L$,
    $\HT_\sigma(\rho)=\HT_\sigma(\rho_0)$, and contains two distinct elements.
  \item
    \begin{itemize}
    \item For all $v|p$, $\rho|_{G_{F_v}}$ and $\rho_0|_{G_{F_v}}$ are crystalline.
    \item $p$ is unramified in $F$.
    \item For all $\sigma:F\into L$, the elements of 
    $\HT_\sigma(\rho)$ differ by at most $p-2$.
    \end{itemize}
  \item $\Im\rhobar\supseteq\SL_2(\F_p)$.

  \end{enumerate}
Then $\rho$ is modular.
\end{thm}

\subsection{The integral theory of automorphic forms}\label{subsec:
  integral auto forms}In order to prove
Theorem~\ref{thm: main modularity lifting theorem}, we will need to
study congruences between automorphic forms. This is easier to do if
we  work with automorphic forms on $G_D(\A^\infty)$,
where $S(D)=S_\infty$. In order to do this, assume that $[F:\Q]$ is
even. (We will reduce to this case by base change.) Then such a $D$
exists, and we have $G_D(\A^\infty)\cong\GL_2(\A_F^\infty)$, and
$(D\otimes_\Q\R)^\times/(F\otimes_\Q\R)^\times$ is compact.

Fix an isomorphism $\imath:\Lbar\isoto\C$, and some $k\in\Z_{\ge
  2}^{\Hom(F,\C)}$, $\eta\in\Z^{\Hom(F,\C)}$ with
$w:=k_\tau+2\eta_\tau-1$ independent of $\tau$. Let
$U=\prod_vU_v\subset\GL_2(\A_F^\infty)$ be a compact open subgroup,
and let $S$ be a finite set of finite places of $F$, not containing any of
the places lying over $p$, with the property that if $v\notin S$, then
$U_v=\GL_2(\cO_{F_v})$.

Let $U_S:=\prod_{v\in S}U_v$, write $U=U_SU^S$, let $\psi:U_S\to\cO^\times$ be a continuous homomorphism
(which implies that it has open kernel), and let
$\chi_0:\A_F^\times/F^\times\to\C^\times$ be an algebraic
Gr\"ossencharacter with the properties that
\begin{itemize}
\item $\chi_0$ is unramified outside $S$,
\item for each place $v|\infty$,
  $\chi_0|_{(F_v^\times)^\circ}(x)= x^{1-w}$, and
\item $\chi_0|_{(\prod_{v\in S}F_v^\times)\cap U_S}=\imath\circ\psi^{-1}$.
\end{itemize}
As in Theorem~\ref{thm: Gr\"ossencharacters and algebraic
    representations}, this gives us a
  character \[\chi_{0,\imath}:\A_F^\times/\overline{F^\times
    (F_\infty^\times)^\circ}\to\Lbar^\times,\] \[x\mapsto\bigl(\prod_{\tau:F\into
    L}\tau(x_p)^{1-w}\bigr)\imath^{-1}(\prod_{\tau:F\into\C}\tau(x_\infty))^{w-1}\chi_0(x).\]

Our spaces of ($p$-adic) algebraic automorphic forms will be defined in a similar
way to the more classical spaces defined in Section~\ref{subsec:
  modular forms}, but with the role of the infinite places being
played by the places lying over $p$. Accordingly, we define
coefficient systems in the following way. Assume that~$L$ is
sufficiently large that it contains the image of~$\chi_{0,\imath}$.

Let
$\Lambda=\Lambda_{k,\eta,\imath}=\otimes_{\tau:F\into\C}\Sym^{k_\tau-2}(\cO^2)\otimes(\wedge^2\cO^2)^{\otimes\eta_\tau}$,
and let $\GL_2(\cO_{F,p}):=\prod_{v|p}\GL_2(\cO_{F_v})$ act on $\Lambda$ via $\imath^{-1}\tau$ on
the $\tau$-factor. In particular,
$\Lambda\otimes_{\cO,\imath}\C\cong\otimes_{\tau:F\into\C}\Sym^{k_\tau-2}(\C^2)\otimes(\wedge^2\C^2)^{\otimes\eta_\tau}$,
which has an obvious action of $\GL_2(F_\infty)$, and the two actions
of $\GL_2(\cO_{F,(p)})$ (via its embeddings into $\GL_2(\cO_{F,p})$
and $\GL_2(F_\infty)$) are compatible.

Let $A$ be a finite $\cO$-module. Since~$D$ is fixed, we drop it from
the notation from now on. We define
$S(U,A)=S_{k,\eta,\imath,\psi,\chi_0}(U,A)$ to be the space of
functions \[\phi:D^\times\backslash
\GL_2(\A_F^\infty)\to\Lambda\otimes_\cO A\] such that for all
$g\in\GL_2(\A_F^\infty),u\in U,z\in(\A_F^\infty)^\times$, we
have \[\phi(guz)=\chi_{0,\imath}(z)\psi(u_S)^{-1}u_p^{-1}\phi(g).\]

Since $D^\times\backslash \GL_2(\A_F^\infty)/U(\A_F^\infty)^\times$ is
finite, we see in particular that $S(U,\cO)$ is a finite free
$\cO$-module. It has a Hecke action in the obvious way: let
$\widetilde{\T}:=\cO[T_v,S_v: v\nmid p,v\notin S]$, let $\varpi_v$ be
a uniformizer of~$F_v$, and let $T_v,S_v$ act
via the usual double coset operators corresponding to $
\begin{pmatrix}
  \varpi_v&0\\0&1
\end{pmatrix}$, $
\begin{pmatrix}
  \varpi_v&0\\0&\varpi_v
\end{pmatrix}$. Let $\T_U$ be the image of $\widetilde{\T}$ in
$\End_{\cO}(S(U,\cO))$, so that $\T_U$ is a commutative $\cO$-algebra
which acts faithfully on $S(U,\cO)$, and is finite free as an
$\cO$-module.

As in~\cite[Lem.\ 1.3]{tay-fm2}, to which we refer for more details,
there is an isomorphism\[S(U,\cO)\otimes_{\cO,\imath}\C\isoto\Hom_{U_S}(\C(\psi^{-1}),S_{k,\eta}^{U^S,\chi_0}),\]with
the map being \[\phi\mapsto (g\mapsto
g_\infty^{-1}\imath(g_p\phi(g^\infty))),\]where $g_p$ acts on
$\Lambda\otimes_{\cO,\imath}\C$ via the obvious extension of the action
of $\GL_2(\cO_{F,(p)})$ defined above, and the target of the
isomorphism is the elements $\phi'\in S_{k,\eta}$ with
$z\phi'=\chi_0(z)\phi'$ for all $z\in(\A_F^\infty)^\times$,
$u\phi'=\psi(u_S)^{-1}\phi'$ for all $u\in U$. This isomorphism is
compatible with the actions of $\widetilde{\T}$ on each side. The target
is isomorphic
to \[\oplus_{\pi}\Hom_{U_S}(\C(\psi^{-1}),\pi_S)\otimes\otimes'_{v\notin
  S}\pi_v^{\GL_2(\cO_{F_v})},\]where the sum is over the cuspidal
automorphic representations $\pi$ of $G_D(\A^\infty)$ of weight
$(k,\eta)$, which have central character $\chi_0$ and are unramified
outside of $S$ (so that in particular, for $v\notin S$,
$\pi_v^{\GL_2(\cO_{F_v})}$ is a one-dimensional $\C$-vector space).

By strong multiplicity one, this means that we have an
isomorphism \[\T_U\otimes_{\cO,\imath}\C\cong\prod_{\pi\text{ as
    above, with }\Hom_{U_S}(\C(\psi^{-1}),\pi_S)\ne0}\C\]sending
$T_v,S_v$ to their eigenvalues on
$\pi_v^{\GL_2(\cO_{F_v})}$. (Note in particular that this shows that
$\T_U$ is reduced.) This shows that there is a bijection
between $\imath$-linear ring homomorphisms $\theta:\T_U\to\C$ and the
set of $\pi$ as above, where $\pi$ corresponds to the character taking
$T_v,S_v$ to their corresponding eigenvalues. 

Each $\pi$ has a corresponding Galois representation. Taking the
product of these representations, we obtain a
representation \[\rhomod:G_F\to\prod_\pi\GL_2(\Lbar)=\GL_2(\T_U\otimes_\cO\Lbar),\]
which is characterised by the properties that it is unramified outside
of $S\cup\{v|p\}$, and for any $v\notin S$, $v\nmid p$, we have
$\tr\rhomod(\Frob_v)=T_v$, $\det\rhomod(\Frob_v)=\#k(v)S_v$.

Let $\m$ be a maximal ideal of $\T_U$. Then if $\gp\subsetneq\m$ is a
minimal prime, then there is an injection $\theta:\T_U/\gp\into\Lbar$,
which corresponds to some $\pi$ as above. (This follows from the
going-up and going-down theorems, and the fact that $\T_U$ is finitely
generated and free over $\cO$.) The semisimple mod $p$
Galois representation corresponding to $\pi$ can be conjugated to give
a representation $\rhobar_\m:G_F\to\GL_2(\T_U/\m)$ (because the trace
and determinant are valued in $\T_U/\m$, which is a finite field, and
thus has trivial Brauer group, so the Schur index is trivial). This is well defined (up to
isomorphism) independently of the choice of $\gp$ and $\theta$ (by the
Chebotarev density theorem).

Since $\T_U$ is finite over the complete local ring $\cO$, it is
semilocal, and we can write $\T_U=\prod_\m\T_{U,\m}$. Suppose now that
$\rhobar_\m$ is absolutely irreducible. Then we have the
representation \[\rhomod_\m:G_F\to\GL_2(\T_{U,\m}\otimes_\cO\Lbar)=\prod_\pi\GL_2(\Lbar),\]where
the product is over the $\pi$ as above with
$\rhobar_{\pi,\imath}\cong\rhobar_\m$. Each representation to
$\GL_2(\Lbar)$ can be conjugated to lie in $\GL_2(\cO_{\Lbar})$, and
after further conjugation (so that the residual representations are
equal to $\rhobar_\m$, rather than just conjugate to it), the image of
$\rhomod_\m$ lies in the subring of $\prod_\pi\GL_2(\cO_{\Lbar})$
consisting of elements whose image modulo the maximal ideal of $\cO_{\Lbar}$
lie in $\T_U/\m$. We can then apply Lemma~\ref{lem: carayol traces
  lemma} to see that $\rhomod_\m$ can be conjugated to lie in
$\GL_2(\T_{U,\m})$. We will write $\rhomod_\m:G_F\to\GL_2(\T_{U,\m})$ for
the resulting representation from now on.

We will sometimes want to consider Hecke operators at places in
$S$. To this end, let $T\subseteq S$ satisfy $\psi|_{U_T}=1$, and
choose $g_v\in\GL_2(F_v)$ for each $v\in T$. Set $W_v=[U_v g_v U_v]$,
and define $\T_U\subseteq\T_U'\subseteq\End_\cO(S(U,\cO))$ by
adjoining the $W_v$ for $v\in T$. This is again commutative, and
finite and flat over $\cO$. However, it need not be reduced; indeed,
we have \[\T_U'\otimes_{\cO,\imath}\C\cong\oplus_\pi\otimes_{v\in
  T}\{\text{ subalgebra of }\End_\C(\pi_v^{U_v})\text{ generated by
}W_v\},\]so that there is a bijection between $\imath$-linear
homomorphisms $\T'_U\to\C$ and tuples $(\pi,\{\alpha_v\}_{v\in T})$,
where $\alpha_v$ is an eigenvalue of $W_v$ on $\pi_v^{U_v}$. (Note that we will
not explicitly use the notation $\T_U'$ again for a Hecke algebra, but that for
example the Hecke algebras $\T_{U_Q}$ used in the patching argument below, which
incorporate Hecke operators at the places in $Q$, are an example of this construction.)

 We can
write \[\GL_2(\A_F^\infty)=\coprod_{i\in I} D^\times g_i
U(\A_F^\infty)^\times\] for some finite indexing set $I$, and so we
have an injection $S(U,A)\into\oplus_{i\in I}(\Lambda\otimes_{\cO} A)$, by
sending $\phi\mapsto (\phi(g_i))$. To determine the image, we need
to consider when we can have $g_i=\delta g_iuz$ for $\delta\in
D^\times$, $z\in(\A_F^\infty)^\times$, $u\in U$ (because then
$\phi(g_i)=\phi(\delta g_i
uz)=\chi_{0,\imath}(z)\psi(u_S)^{-1}u_p^{-1}\phi(g_i)$). We see in
this way that we obtain an isomorphism\[S(U,A)\isoto \oplus_{i\in
  I}(\Lambda\otimes A)^{(U(\A_F^\infty)^\times\cap g_i^{-1}D^\times
  g_i)/F^\times}.\]

We need to have some control on these finite groups $G_i:=(U(\A_F^\infty)^\times\cap g_i^{-1}D^\times
  g_i)/F^\times$. (Note that they are finite, because $D^\times$ is
  discrete in $G_D(\A^\infty)$.)
   Since we have assumed that $p>3$ and $p$ is
  unramified in $F$, we see that $[F(\zeta_p):F]>2$. Then we claim that $G_i$ has order prime to $p$. To see this, note that if
  $g_i^{-1}\delta g_i$ is in this group, with $\delta\in D^\times$,
  then $\delta^2/\det\delta \in D^\times \cap g_iUg_i^{-1}(\det U)$,
  the intersection of a discrete set and a compact set, so
  $\delta^2/\det\delta$ has finite order, i.e. is a root of
  unity. However any element of $D$ generates an extension of $F$ of
  degree at most $2$, so by the assumption that $[F(\zeta_p):F]>2$, it
  must be a root of unity of degree prime to $p$, and there is some
  $p\nmid N$ with $\delta^{2N}\in F^\times$, so that $g_i^{-1}\delta
  g_i$ has order prime to $p$, as required.
  \begin{prop}
    \label{prop: freeness of modular forms over Hecke algebra}\leavevmode
    \begin{enumerate}
    \item We have $S(U,\cO)\otimes_\cO A\isoto S(U,A)$.
    \item If $V$ is an open normal subgroup of $U$ with $\#(U/V)$ a
      power of $p$, then $S(V,\cO)$ is a free
      $\cO[U/V(U\cap(\A_F^\infty)^\times)]$-module.
    \end{enumerate}

  \end{prop}
  \begin{proof}
    (1) This is immediate from the isomorphism $S(U,A)\isoto \oplus_{i\in
  I}(\Lambda\otimes A)^{G_i}$, because the fact that the $G_i$ have
order prime to $p$ means that $(\Lambda\otimes
A)^{G_i}=(\Lambda)^{G_i}\otimes A$. 

(2) Write $U=\coprod_{j\in J} u_j V(U\cap (\A_F^\infty)^\times)$. We
claim that we have $\GL_2(\A_F^\infty)=\coprod_{i\in I,j\in J}
D^\times g_i u_j V(\A_F^\infty)^\times$, from which the result is
immediate. To see this, we need to show that if $g_iu_j=\delta
g_{i'}u_{j'}vz$ then $i=i'$ and $j=j'$. 

That $i=i'$ is immediate from the definition of $I$, so we have
$u_{j'}vu_j^{-1} z=g_i^{-1}\delta^{-1}g_i$. As above, there is some
positive integer $N$ coprime to $p$ such that $\delta^N\in F^\times$,
so $(u_{j'}vu_j^{-1})^N\in(\A_F^\infty)^\times$. Since $V$ is normal
in $U$, we can write $(u_{j'}vu_j^{-1})^N=(u_{j'}u_j^{-1})^Nv'$ for
some $v'\in V$, so that $(u_{j'}u_j^{-1})^N\in
V(U\cap(\A_F^\infty)^\times)$. Since $\#(U/V)$ is a power of $p$, we
see that in fact $u_{j'}u_j^{-1}\in
V(U\cap(\A_F^\infty)^\times)$, so that $j=j'$ by the definition of $J$.
  \end{proof}

\subsection{Base change}We begin the proof of Theorem~\ref{thm: main modularity lifting theorem} by using base change to reduce to a
special case. By Facts~\ref{fact: cyclic base change} and~\ref{fact:
  existence of number fields with given local properties}, we can
replace $F$ by a solvable totally real extension which is unramified
at all primes above $p$, and assume that
\begin{itemize}
\item $[F:\Q]$ is even.
\item $\rhobar$ is unramified outside $p$.
\item For all places $v\nmid p$, both $\rho(I_{F_v})$ and
  $\rho_0(I_{F_v})$ are unipotent (possibly trivial).
\item If $\rho$ or $\rho_0$ are ramified at some place $v\nmid p$,
  then $\rhobar|_{G_{F_v}}$ is trivial, and $\#k(v)\equiv 1\pmod{p}$.
\item $\det\rho=\det\rho_0$. [To see that we can assume this, note that the assumption
  that $\rho,\rho_0$ are crystalline with the same Hodge--Tate weights
  for all places dividing $p$ implies that $\det\rho/\det\rho_0$ is
  unramified at all places dividing $p$. Since we have already assumed
  that  $\rho(I_{F_v})$ and
  $\rho_0(I_{F_v})$ are unipotent for all places $v\nmid p$, we see
  that the character $\det\rho/\det\rho_0$ is unramified at all places, and thus
  has finite order. Since it is residually trivial, it has $p$-power
  order, and is thus trivial on all complex conjugations; so the
  extension cut out by its kernel is a finite, abelian, totally real
  extension which is unramified at all places dividing $p$.]
\end{itemize}
We will assume from now on that all of these conditions hold. Write
$\chi$ for $\det\rho=\det\rho_0$; then we have
$\chi\varepsilon_p=\chi_{0,\imath}$ for some algebraic
Gr\"ossencharacter $\chi_0$.


From now on, we will assume without further comment that the
coefficient field $L$ is sufficiently large, in the
sense that $L$ contains a primitive $p$-th root of unity, and for all
$g\in G_F$, $\F$ contains the eigenvalues of $\rhobar(g)$.

\subsection{Patching} Having used base change to impose the additional
conditions of the previous section, we are now in a position to begin
the main patching argument. 

We let $D/F$ be a quaternion algebra ramified at exactly the infinite
places (which exists by our assumption that $[F:\Q]$ is even). By the
Jacquet--Langlands correspondence, we can and will work with
automorphic representations of $G_D(\A^\infty)$ from now on.

Let $T_p$ be the set of places of $F$ lying over $p$, let $T_r$ be the
set of places not lying over $p$ at which $\rho$ or $\rho_0$ is
ramified, and let $T=T_p\coprod T_r$. If $v\in T_r$, write $\sigma_v$
for a choice of topological generator of $I_{F_v}/P_{F_v}$. By our
assumptions above, if $v\in T_r$ then $\rhobar|_{G_{F_v}}$ is trivial,
$\rho|_{I_{F_v}},\rho_0|_{I_{F_v}}$ are unipotent, and $\#k(v)\equiv
1\pmod{p}$.

The patching argument will involve the consideration of various finite
sets $Q$ of auxiliary finite places. We will always assume that if
$v\in Q$, then
\begin{itemize}
\item $v\notin T$,
\item $\#k(v)\equiv 1\pmod{p}$, and
\item $\rhobar(\Frob_v)$ has distinct eigenvalues, which we denote
  $\alphabar_v$ and $\betabar_v$.
\end{itemize}

For each set $Q$ of places satisfying these conditions, we define
deformation problems $\cS_Q=(T\cup Q,\{\cD_v\},\chi)$ and
$\cS'_Q=(T\cup Q,\{\cD'_v\},\chi)$ as follows. (The reason for
considering both problems is that the objects without a prime are the
ones that we ultimately wish to study, but the objects with a prime
have the advantage that the ring $(R^{\operatorname{loc},'})^\red$ defined
below is irreducible. We will exploit this irreducibility, and the
fact that the two deformation problems agree modulo~$p$.) Let $\zeta$ be a fixed
primitive $p$-th root of unity in $L$.
\begin{itemize}
\item If $v\in T_p$, then $\cD_v=\cD'_v$ is chosen so that
  $R_{\rhobar|_{G_{F_v}},\chi}^{\square}/I(\cD_v)=R^\square_{\rhobar|_{G_{F_v}},\chi,\cris,\{\HT_\sigma(\rho)\}}$.
\item If $v\in Q$, then $\cD_v=\cD_v'$ consists of all lifts of
  $\rhobar|_{G_{F_v}}$ with determinant $\chi$.
\item If $v\in T_r$, then $\cD_v$ consists of all lifts of
  $\rhobar|_{G_{F_v}}$ with $\chara_{\rho(\sigma_v)}(X)=(X-1)^2$,
  while $\cD'_v$ consists of all lifts with
  $\chara_{\rho(\sigma_v)}(X)=(X-\zeta)(X-\zeta^{-1})$.
\end{itemize}(In particular, the difference between $\cS_Q$ and $\cS_\emptyset$
is that we have allowed our deformations to ramify at places in $Q$.)
We write \[\Rloc=\cotimes_{v\in
  T,\cO}R^\square_{\rhobar|_{G_{F_v}},\chi}/I(\cD_v),\
R^{\operatorname{loc},'}=\cotimes_{v\in
  T,\cO}R^\square_{\rhobar|_{G_{F_v}},\chi}/I(\cD'_v).\] Then
$\Rloc/\lambda=R^{\operatorname{loc},'}/\lambda$, because $\zeta\equiv
1\pmod{\lambda}$. In addition,  we see from Theorems~\ref{thm:FL deformation rings} and~\ref{thm:Ihara avoidance deformations} that
\begin{itemize}
\item $(R^{\operatorname{loc},'})^\red$ is irreducible, $\cO$-flat, and has Krull
  dimension $1+3\#T+[F:\Q]$,
\item $(\Rloc)^\red$ is $\cO$-flat, equidimensional of Krull
  dimension $1+3\#T+[F:\Q]$, and reduction modulo $\lambda$ gives a
  bijection between the irreducible components of $\Spec \Rloc$ and
  those of $\Spec\Rloc/\lambda$.
\end{itemize}
We have the global analogues $R_Q^\univ:=R_{\rhobar,\cS_Q}^\univ$,
$R_Q^{\univ,'}:=R_{\rhobar,\cS'_Q}^\univ$,
$R_Q^\square:=R_{\rhobar,\cS_Q}^{\square_T}$,
$R_Q^{\square,'}:=R_{\rhobar,\cS'_Q}^{\square_T}$, and we have
$R_Q^\univ/\lambda=R_Q^{\univ,'}/\lambda$,
$R_Q^\square/\lambda=R_Q^{\square,'}/\lambda$. There are obvious
natural maps $\Rloc\to R_Q^\square$, $R^{\operatorname{loc},'}\to R_Q^{\square,'}$, and
these maps agree after reduction mod $\lambda$.

We can and do fix representatives $\rho_Q^\univ,\rho_Q^{\univ,'}$ for
the universal deformations of $\rhobar$ over $R_Q^\univ,R_Q^{\univ,'}$
respectively, which are compatible with the choices of
$\rho_\emptyset^\univ,\rho_\emptyset^{\univ,'}$, and so that the induced
surjections \[R_Q^\univ\onto R_\emptyset^\univ,\ R_Q^{\univ,'}\onto
R_\emptyset^{\univ,'}\] are identified modulo $\lambda$.

Fix a place $v_0\in T$, and set $\cJ:=\cO\llbracket X_{v,i,j}\rrbracket _{v\in T,
  i,j=1,2}/(X_{v_0,1,1})$. Let $\ca$ be the ideal of $\cJ$ generated
by the $X_{v,i,j}$. Then our choice of $\rho_Q^\univ$ gives an
identification $R_Q^\square\isoto R_Q^\univ\cotimes_\cO\cJ$,
corresponding to the universal $T$-framed deformation
$(\rho_Q^\univ,\{1+(X_{v,i,j})\}_{v\in T})$.

Now, by Exercise~\ref{ex: TW deformations}, for each place $v\in Q$ we have an
isomorphism $\rho_Q^\univ|_{G_{F_v}}\cong\chi_\alpha\oplus\chi_\beta$,
where $\chi_\alpha,\chi_\beta:G_{F_v}\to (R_Q^\univ)^\times$, where
$(\chi_\alpha\text{ mod }\m_{R_Q^\univ})(\Frob_v)=\alphabar_v$,
$(\chi_\beta\text{ mod }\m_{R_Q^\univ})(\Frob_v)=\betabar_v$. 

Let $\Delta_v$ be the maximal $p$-power quotient of
$k(v)^\times$ (which we sometimes regard as a subgroup of~$k(v)^\times$). Then $\chi_\alpha|_{I_{F_v}}$ factors through the
composite \[I_{F_v}\onto I_{F_v}/P_{F_v}\onto k(v)^\times\onto
\Delta_v,\]and if we write
$\Delta_Q=\prod_{v\in Q}\Delta_v$,
$(\prod\chi_\alpha):\Delta_Q\to(R_Q^\univ)^\times$, then we see that
$(R_Q^\univ)_{\Delta_Q}=R_\emptyset^\univ$.

The isomorphism $R_Q^\square\isoto R_Q^\univ\cotimes_\cO\cJ$ and the
homomorphism $\Delta_Q\to(R_Q^\univ)^\times$ together give a
homomorphism $\cJ[\Delta_Q]\to R_Q^\square$. In the same way, we have
a homomorphism $\cJ[\Delta_Q]\to R_Q^{\square,'}$, and again these
agree modulo $\lambda$. If we write
$\ca_Q:=\langle\ca,\delta-1\rangle_{\delta\in\Delta_Q}\lhd\cJ[\Delta_Q]$,
then we see that $R_Q^\square/\ca_Q=R_\emptyset^\univ$, and that
$R_Q^{\square,'}/\ca_Q=R_\emptyset^{\univ,'}$, and again these agree
modulo $\lambda$.

We now examine the spaces of modular forms that we will patch. We have
our fixed isomorphism $\imath:\overline{L}\isoto\C$, and an algebraic
Gr\"ossencharacter $\chi_0$ such that
$\chi\varepsilon_p=\chi_{0,\imath}$. Define $k,\eta$ by
$\HT_\tau(\rho_0)=\{\eta_{\imath\tau},\eta_{\imath\tau}+k_{\imath\tau}-1\}$. We
define compact open subgroups $U_Q=\prod U_{Q,v}$, where:
\begin{itemize}
\item $U_{Q,v}=\GL_2(\cO_{F_v})$ if $v\notin Q\cup T_r$,
\item $U_{Q,v}=U_0(v)=\{
  \begin{pmatrix}
    *&*\\0&*
  \end{pmatrix}\pmod{v}\}$ if $v\in T_r$, and
\item $U_{Q,v}=\{ \begin{pmatrix}
    a&b\\c&d
  \end{pmatrix}\in U_0(v)|a/d\pmod{v}\in k(v)^\times\mapsto
  1\in\Delta_v\}$ if $v\in Q$. 

\end{itemize}
We let $\psi:\prod_{v\in
  Q\cup T_r}U_{Q,v}\to\cO^\times$ be the trivial character. Similarly, we
set $U'_Q=U_Q$, and we define  $\psi':\prod_{v\in
  Q\cup T_r}U_{Q,v}\to\cO^\times$ in the following way.  For each $v\in T_r$,
we have a homomorphism $U_{Q,v}\to k(v)^\times$ given by sending $
\begin{pmatrix}
  a&b\\c&d
\end{pmatrix}$ to $a/d\pmod{v}$, and we compose these characters with
the characters $k(v)^\times\to\cO^\times$ sending the image of
$\sigma_v$ to $\zeta$,
where $\sigma_v$ is a generator of $I_{F_v}/P_{F_v}$. We let~$\psi'$
be trivial at the places in~$Q$.

We obtain spaces of modular forms $S(U_Q,\cO)$, $S(U'_Q,\cO)$ and
corresponding Hecke algebras $\T_{U_Q}$, $\T_{U'_Q}$, generated by the
Hecke operators $T_v,S_v$ with $v\notin T\cup Q$, together with Hecke
operators~$U_{\varpi_v}$ for~$v\in Q$ (depending on a chosen uniformiser~$\varpi_v$) defined by \[U_{\varpi_v}=\left[U_{Q,v}
  \begin{pmatrix}
    \varpi_v&0\\0&1
  \end{pmatrix}U_{Q,v}\right].\] Note that
$\psi=\psi'\pmod{\lambda}$, so we have
$S(U_\emptyset,\cO)/\lambda=S(U'_\emptyset,\cO)/\lambda$. We let $\m_\emptyset\lhd
\T_{U_\emptyset}$ be the ideal generated by $\lambda$ and the
$\tr\rhobar(\Frob_v)-T_v$, $\det\rhobar(\Frob_v)-\#k(v)S_v$, $v\notin
T$. This is a maximal ideal of $\T_{U_\emptyset}$, because 
it is the kernel of the homomorphism $\T_{U_\emptyset}\to\cO\onto\F$,
where the map $\T_{U_\emptyset}\to\cO$ is the one coming from the
automorphicity of $\rho_0$, sending $T_v\mapsto\tr\rho_0(\Frob_v)$,
$S_v\mapsto\#k(v)^{-1}\det\rho_0(\Frob_v)$.

Write $\T_\emptyset:=\T_{U_\emptyset,\m_\emptyset}$. We have a lifting $\rhomod:G_F\to\GL_2(\T_\emptyset)$ of
type $\cS_\emptyset$, so by the universal property of $R_\emptyset^\univ$, we have a surjection
$R_\emptyset^\univ\onto\T_\emptyset$ (it is surjective because
local-global compatibility shows that the Hecke operators generating
$\T_{\emptyset}$ are in the image).
Similarly, we have a surjection
$R_\emptyset^{\univ,'}\onto\T'_\emptyset:=\T_{U'_\emptyset,\m_\emptyset}$. Set
$S_\emptyset:=S(U_\emptyset,\cO)_{\m_\emptyset}$,
$S'_\emptyset:=S(U'_\emptyset,\cO)_{\m_\emptyset}$. Then
 the identification $R^\univ_\emptyset/\lambda\cong
R^{\univ,'}_\emptyset/\lambda$ is compatible with
$S_\emptyset/\lambda=S'_\emptyset/\lambda$.

\begin{lem} If
$\Supp_{R_\emptyset^\univ}(S_\emptyset)=\Spec
R_\emptyset^\univ$, then  $\rho$ is modular.  
\end{lem}
\begin{proof}
  Suppose that $\Supp_{R_\emptyset^\univ}(S_\emptyset)=\Spec
R_\emptyset^\univ$. Since $S_\emptyset$ is a faithful
  $\T_\emptyset$-module by definition, we see that
  $\ker(R_\emptyset^\univ\to\T_\emptyset)$ is nilpotent, so that
  $(R_\emptyset^\univ)^\red\isoto\T_\emptyset$. Then $\rho$
  corresponds to some homomorphism $R^\univ_\emptyset\to\cO$, and thus
  to a homomorphism $\T_\emptyset\to\cO$, and the composite of this
  homomorphism with $\imath:\cO\into\C$ corresponds to a cuspidal
  automorphic representation $\pi$ of $G_D(\A^\infty)$ of weight
  $(k,\eta)$, which by construction has the property that
  $\rho\cong\rho_{\pi,\imath}$, as required.
\end{proof}

To show that $\Supp_{R_\emptyset^\univ}(S_\emptyset)=\Spec
R_\emptyset^\univ$, we will study the above constructions as $Q$
varies. Let $\m_Q\lhd\T_{U_Q}$ be the maximal ideal generated by
$\lambda$, the $\tr\rhobar(\Frob_v)-T_v$ and
$\det\rhobar(\Frob_v)-\#k(v)S_v$ for $v\notin T\cup Q$, and the
$U_{\varpi_v}-\alphabar_v$ for $v\in Q$.

We write $S_Q=S_{U_Q}:=S(U_Q,\cO)_{\m_Q}$, and
$\T_Q:=(\T_{U_Q})_{\m_Q}$. We have a homomorphism
$\Delta_Q\to\End(S_Q)$, given by sending $\delta\in\Delta_v$ to $
\begin{pmatrix}
  \delta&0\\0&1
\end{pmatrix}\in U_0(v)$. We also have another homomorphism
$\Delta_Q\to\End(S_Q)$, given by the composite \[\Delta_Q\to
R_Q^\univ\onto \T_Q\to\End(S_Q).\] Let $U_{Q,0}:=\prod_{v\notin Q}U_{Q,v}\prod_{v\in Q}U_0(v)$. Then
$U_Q$ is a normal subgroup of $U_{Q,0}$, and $U_{Q,0}/U_Q=\Delta_Q$.

We now examine the consequences of local-global compatibility at the
places in $Q$. 

\begin{prop}
  \label{fact: two actions of Delta agree}
  \begin{enumerate}
  \item The two homomorphisms $\Delta_Q\to\End(S_Q)$ (the other one
    coming via $R_Q^\univ$) are equal. 
  \item  $S_Q$ is finite free over
    $\cO[\Delta_Q]$.
  \end{enumerate}
\end{prop}
\begin{proof}
  A homomorphism $\theta:\T_Q\to\overline{L}\isoto\C$ corresponds to a cuspidal 
  automorphic representation $\pi$, and for each $v\in Q$ the image
  $\alpha_v$ of $U_{\varpi_v}$ is such that $\alpha_v$ is an
  eigenvalue of $U_{\varpi_v}$ on $\pi_v^{U_{Q,v}}$.

  It can be checked that since $\pi_v^{U_{Q,v}}\ne 0$, $\pi_v$ is
  necessarily a subquotient of $\chi_1\times \chi_2$ for some tamely
  ramified characters $\chi_1,\chi_2:F_v^\times\to\C^\times$. Then one
  checks explicitly
  that \[(\chi_1\times\chi_2)^{U_{Q,v}}\cong\C\phi_1\oplus\C\phi_w,\]
  where $w=
  \begin{pmatrix}
    0&1\\1&0
  \end{pmatrix}$, $\phi_1(1)=\phi_w(w)=1$, and
  $\Supp\phi_1=B(F_v)U_{Q,v}$, $\Supp\phi_w=B(F_v)wU_{Q,v}$.

  Further explicit calculation shows
  that
  \[U_{\varpi_v}\phi_1=\#k(v)^{1/2}\chi_1(\varpi_v)\phi_1+ X\phi_w\]
  for some $X$, which is $0$ if $\chi_1/\chi_2$ is ramified,
  and \[U_{\varpi_v}\phi_w=\#k(v)^{1/2}\chi_2(\varpi_v)\phi_w.\] Note
  that by local-global compatibility,
  $\imath^{-1}(\#k(v)^{1/2}\chi_1(\varpi_v))$ and
  $\imath^{-1}(\#k(v)^{1/2}\chi_2(\varpi_v))$ are the eigenvalues of
  $\rho_{\pi,\imath}(\Frob_v)$, so one of them is a lift of
  $\alphabar_v$, and one is a lift of $\betabar_v$.   As
  a consequence, we see that $\chi_1/\chi_2\ne |\cdot|^{\pm 1}$ (as if
  this equality held, we would have
  $\alphabar_v/\betabar_v\equiv \#k(v)^{\pm 1}\equiv 1\pmod{\lambda}$,
  contradicting our assumption that
  $\alphabar_v\ne\betabar_v$). Consequently we have
  $\pi_v=\chi_1\times\chi_2\cong\chi_2\times\chi_1$, so that  without
  loss of generality we have
  $\chibar_1(\varpi_v)=\betabar_v, \chibar_2(\varpi_v)=\alphabar_v$.

  It is also easily checked that \[
    \begin{pmatrix}
      \delta&0\\0&1
    \end{pmatrix}\phi_1=\chi_1(\delta)\phi_1,\ \begin{pmatrix}
      \delta&0\\0&1
    \end{pmatrix}\phi_w=\chi_2(\delta)\phi_w.\]


  We see that
  $S_Q\otimes_{\cO,\imath}\C=\oplus_\pi\otimes_{v\in Q}X_v$, where
  $X_v$ is the 1-dimensional space where $U_{\varpi_v}$ acts via a
  lift of $\alphabar_v$. Since this space is spanned by $\phi_w$, we
  see that $\Delta_v$ acts on $S_Q$ via
  $\chi_2=\chi_\alpha\circ\Art$. This completes the proof of the first part.

  Finally, the second part is  immediate from Proposition~\ref{prop: freeness of modular forms over Hecke algebra}(2).
\end{proof}

 Fix a place $v\in Q$. Since $\alphabar_v\ne\betabar_v$, by Hensel's
 lemma we may write $\chara\rhomod_\emptyset(\Frob_v)=(X-A_v)(X-B_v)$
 for some $A_v,B_v\in\T_\emptyset$ with $A_v\equiv\alphabar_v,
 B_v\equiv\betabar_v\pmod{\m_\emptyset}$. 
\begin{prop}
  \label{prop: what fixing U_v eigenvalues does}We have an
  isomorphism
  $\prod_{v\in Q}(U_{\varpi_v}-B_v):S_\emptyset\isoto
  S(U_{Q,0},\cO)_{\m_Q}$ (with the morphism being defined by viewing
  the  source and target as submodules of $S(U_{Q,0},\cO)_{\m_\emptyset}$).
\end{prop}
\begin{proof}
  We claim that it is enough to prove that the map is an isomorphism
  after tensoring with $L$, and an injection after tensoring with
  $\F$. To see this, write $X:=S_\emptyset, Y:=S(U_{Q,0},\cO)_{\m_Q}$,
  and write $Q$ for the cokernel of the map $X\to Y$. Then $X,Y$ are
  finite free $\cO$-modules, and if the map $X\otimes L\to Y\otimes L$
  is injective, then so is the map $X\to Y$, so that we have a short
  exact sequence $0\to X\to Y\to Q\to 0$. Tensoring with $L$, we have
  $Q\otimes L=0$. Tensoring with $\F$, we
  obtain an exact sequence $0\to Q[\lambda]\to X\otimes\F\to
  Y\otimes\F\to Q\otimes\F\to 0$, so we have $Q[\lambda]=0$. Thus
  $Q=0$, as required.

  In order to check that we have an isomorphism after tensoring with
  $L$, it is enough to check that the induced map $\prod_{v\in
    Q}(U_{\varpi_v}-B_v):S_\emptyset\otimes_{\cO,\imath}\C\to
  S(U_{Q,0},\cO)_{\m_Q}\otimes_{\cO,\imath}\C$ is an isomorphism. This
  is easily checked: $S_\emptyset\otimes\C\cong\oplus_\pi\otimes_{v\in
    Q}(\chi_{1,v}\times\chi_{2,v})^{\GL_2(\cO_{F_v})}$, and
  $(\chi_{1,v}\times\chi_{2,v})^{\GL_2(\cO_{F_v})}=\C\phi_0$,
  where $\phi_0$ is as in Exercise~\ref{ex:computing Hecke evals on unram PS}(3). Similarly,
  $S(U_{Q,0},\cO)_{\m_Q}\otimes_{\cO,\imath}\C=\oplus_\pi\otimes_{v\in
    Q}M_v$, where $M_v$ is the subspace of $(\chi_{1,v}\times
  \chi_{2,v})^{U_0(v)}$ on which $U_{\varpi_v}$ acts via a lift of
  $\alphabar_v$, which is spanned by $\phi_w$. Since the natural map
  $(\chi_{1,v}\times\chi_{2,v})^{\GL_2(\cO_{F_v})}\to (\chi_{1,v}\times
  \chi_{2,v})^{U_0(v)}$ sends $\phi_0\mapsto \phi_1+\phi_w$ (as
  $\phi_0(1)=\phi_0(w)=1$), the result follows.

It remains to check injectivity after tensoring with $\F$. The kernel
of the map, if nonzero, would be a nonzero finite module for the Artinian local
ring $\T_\emptyset/\lambda$, and would thus have nonzero
$\m_\emptyset$-torsion, so it suffices to prove that the induced
map   \[\prod_{v\in
  Q}(U_{\varpi_v}-B_v):(S_\emptyset\otimes\F)[\m_\emptyset]\to
S(U_{Q,0},\cO)_{\m_Q}\otimes\F\]is an injection. By induction on
$\#Q$, it suffices to prove this in the case that $Q=\{v\}$. Suppose
for the sake of contradiction that there is a nonzero
$x\in(S_\emptyset\otimes\F)[\m_\emptyset]$ with
$(U_{\varpi_v}-\betabar_v)x=0$. Since $x\in S_\emptyset\otimes\F$, we also have
$T_vx=(\alphabar_v+\betabar_v)x$, and we will show that these two
equations together lead to a contradiction. 

Now, $x$ is just a function
$D^\times\backslash\GL_2(\A_F^\infty)\to\Lambda\otimes\F$, on which
$\GL_2(\A_F^\infty)$ acts by right translation. If we make the action
of the Hecke operators explicit, we find that there are $g_i$ such
that $U_v=\coprod_i g_iU_{Q,v}$, and $T_v=\left(\coprod_i g_i\GL_2(\cO_{F_v})\right)\coprod
\begin{pmatrix}
  1&0\\0&\varpi_v
\end{pmatrix}\GL_2(\cO_{F_v})$, so that we have $
\begin{pmatrix}
  1&0\\0&\varpi_v
\end{pmatrix}x=T_vx-U_{\varpi_v}x=\alphabar_vx$. Then $
\begin{pmatrix}
  \varpi_v&0\\0&1
\end{pmatrix}x=w
\begin{pmatrix}
  1&0\\0&\varpi_v
\end{pmatrix}wx=\alphabar_v x$, and $U_{\varpi_v}x=\sum_{a\in k(v)}
\begin{pmatrix}
  \varpi_v&a\\0&1
\end{pmatrix}x=\sum_{a\in k(v)}
\begin{pmatrix}
  1&a\\0&1
\end{pmatrix}
\begin{pmatrix}
  \varpi_v&0\\0&1
\end{pmatrix}
x=\#k(v)\alphabar_vx=\alphabar_vx$. But $U_{\varpi_v}x=\betabar_vx$, so
$\alphabar_v=\betabar_v$, a contradiction.
\end{proof}
Set $S_Q^\square:=S_Q\otimes_{R_Q^\univ}R_Q^\square$. Then we have
$S_Q^\square/\ca_Q=S(U_{Q,0},\cO)_{\m_Q}\isoto S_\emptyset$,
compatibly with the isomorphism $R_Q^\square/\ca_Q\isoto
R_\emptyset^\univ$. Also, $S_Q^\square$ is finite free over
$\cJ[\Delta_Q]$. 

We now return to the Galois side. By Proposition~\ref{prop:H^1 and H^2
  of global over local}, we can and do choose a
presentation \[\Rloc\llbracket x_1,\dots,x_{h_Q}\rrbracket \onto R_Q^{\square},\]where
$h_Q=\#T+\#Q-1-[F:\Q]+\dim_\F H^1_Q(G_{F,T},(\ad^0\rhobar)(1))$, and
$H^1_Q(G_{F,T},(\ad^0\rhobar)(1))=\ker(
H^1(G_{F,T},(\ad^0\rhobar)(1))\to\oplus_{v\in
  Q}H^1(G_{k(v)},(\ad^0\rhobar)(1))$.

The following result will provide us with the sets $Q$ that we will
use. 
\begin{prop}
  \label{prop: Chebotarev argument for TW primes} Let $r=\max(\dim
  H^1(G_{F,T},(\ad^0\rhobar)(1)),1+[F:\Q]-\#T)$. For each $N\ge 1$,
  there exists a set $Q_N$ of places of $F$ such that
  \begin{itemize}
  \item $Q_N\cap T=\emptyset$.
  \item If $v\in Q_N$, then $\rhobar(\Frob_v)$ has distinct
    eigenvalues $\alphabar_v\ne\betabar_v$.
  \item If $v\in Q_N$, then $\#k(v)\equiv 1\pmod{p^N}$.
  \item $\# Q_N=r$.
  \item $R_{Q_N}^{\square}$ (respectively $R_{Q_N}^{\square,'}$) is
    topologically generated over $\Rloc$ (respectively
    $R^{\operatorname{loc},'}$) by $\#T-1-[F:\Q]+r$ elements.
  \end{itemize}
\end{prop}
\begin{proof}
  The last condition may be replaced by
  \begin{itemize}
  \item $H^1_{Q_N}(G_{F,T},(\ad^0\rhobar)(1))=0$.
  \end{itemize}
Therefore, it is enough to show that for each $0\ne [\phi]\in
H^1(G_{F,T},(\ad^0\rhobar)(1))$, there are infinitely many $v\notin T$
such that
\begin{itemize}
\item $\# k(v)\equiv 1\pmod{p^N}$.
\item $\rhobar(\Frob_v)$ has distinct eigenvalues $\alphabar_v, \betabar_v$.
\item $\Res [\phi]\in H^1(G_{k(v)},(\ad^0\rhobar)(1))$ is nonzero.
\end{itemize}
(This then gives us some set of places $Q$ with the given properties,
except that $\#Q$ may be too large; but then we can pass to a subset
of cardinality $r$, while maintaining the injectivity of the map $H^1(G_{F,T},(\ad^0\rhobar)(1))\to\oplus_{v\in
  Q}H^1(G_{k(v)},(\ad^0\rhobar)(1))$.

We will use the Chebotarev density theorem to do this; note that the condition
that $\#k(v)\equiv 1\pmod{p^N}$ is equivalent to $v$ splitting
completely in $F(\zeta_{p^N})$, and the condition that $\rhobar(\Frob_v)$
has distinct eigenvalues is equivalent to asking that
$\ad\rhobar(\Frob_v)$ has an eigenvalue not equal to $1$.

Set $E=\overline{F}^{\ker\ad\rhobar}(\zeta_{p^N})$. We claim that we have
$H^1(\Gal(E/F),(\ad^0\rhobar)(1))=0$. In order to see this, we
claim firstly that $\zeta_p\notin\overline{F}^{\ker\ad\rhobar}$. This
follows from the classification of finite subgroups of
$\PGL_2(\Fpbar)$: we have assumed that $\im\rhobar\supseteq\SL_2(\Fp)$,
and this implies that $\im\ad\rhobar=\PGL_2(\F_{p^s})$ or
$\PSL_2(\F_{p^s})$ for some $s$, and in particular
$(\im\ad\rhobar)^{\ab}$ is trivial or cyclic of order $2$. Since $p\ge
5$ and $p$ is unramified in $F$, we have $[F(\zeta_p):F]\ge 4$, so
$\zeta_p\notin\overline{F}^{\ker\ad\rhobar}$, as claimed.

The extension $E/\overline{F}^{\ker\ad\rhobar}$ is abelian, and we let
$E_0$ be the intermediate field such that $\Gal(E/E_0)$ has order
prime to $p$, while $\Gal(E_0/\overline{F}^{\ker\ad\rhobar})$ has
$p$-power order. Write $\Gamma_1=\Gal(E_0/F)$,
$\Gamma_2=\Gal(E/E_0)$. Then the inflation-restriction exact sequence is in
part \[0\to H^1(\Gamma_1,(\ad^0\rhobar)(1)^{\Gamma_2}) \to
H^1(\Gal(E/F),(\ad^0\rhobar)(1))\to
H^1(\Gamma_2,(\ad^0\rhobar)(1))^{\Gamma_1},\] so in order to show that
$H^1(\Gal(E/F),(\ad^0\rhobar)(1))=0$, it suffices to prove that
$H^1(\Gamma_1,(\ad^0\rhobar)(1)^{\Gamma_2})=
H^1(\Gamma_2,(\ad^0\rhobar)(1))^{\Gamma_1}=0$. 

In fact, we claim that $(\ad^0\rhobar)(1)^{\Gamma_2}$ and
$H^1(\Gamma_2,(\ad^0\rhobar)(1))$ both vanish. For the first of these,
note that $\Gamma_2$ acts trivially on $\ad^0\rhobar$ (since $E_0$
contains $\overline{F}^{\ker\ad\rhobar}$), but that $\zeta_p\notin
E_0$ (as $[E_0:\overline{F}^{\ker\ad\rhobar}]$ is a power of $p$). For
the second term, note that $\Gamma_2$ has prime-to-$p$
order.

Suppose that $\#k(v)\equiv 1\pmod{p}$, and that $\rhobar(\Frob_v)=
\begin{pmatrix}
  \alphabar_v&0\\0&\betabar_v
\end{pmatrix}$. Then $\ad^0\rhobar$ has the basis $
\begin{pmatrix}
  1&0\\0&-1
\end{pmatrix}$, $
\begin{pmatrix}
  0&1\\0&0
\end{pmatrix}$, $
\begin{pmatrix}
  0&0\\1&0
\end{pmatrix}$ of eigenvectors for $\Frob_v$, with eigenvalues $1$,
$\alphabar_v/\betabar_v$, $\betabar_v/\alphabar_v$
respectively. Consequently, we see that there is an isomorphism
$H^1(G_{k(v)},(\ad^0\rhobar)(1))\cong\F$ (since in general for a
(pro)cyclic group, the first cohomology is given by passage to
coinvariants), which we can write explicitly as
$[\phi]\mapsto\pi_v\circ\phi(\Frob_v)\circ i_v$, where $i_v$ is the
injection of $\F$ into the $\alphabar_v$-eigenspace of $\Frob_v$, and
$\pi_v$ is the $\Frob_v$-equivariant projection onto that subspace. 

Let $\sigma_0$ be an element of $\Gal(E/F)$ such that
\begin{itemize}
\item $\sigma_0(\zeta_{p^N})=\zeta_{p^N}$.
\item $\rhobar(\sigma_0)$ has distinct eigenvalues $\alphabar$, $\betabar$.
\end{itemize} (To see that such a $\sigma_0$ exists, note that
$\Gal(\overline{F}^{\ker\rhobar}/F(\zeta_{p^N})\cap\overline{F}^{\ker\rhobar})$
contains $\PSL_2(\Fp)$, and so we can choose $\sigma_0$ so that its
image in this group is an element whose adjoint has an eigenvalue
other than $1$.)
Let $\tE/E$ be the extension cut out by all the $[\phi]\in
H^1(G_{F,T},(\ad^0\rhobar)(1))$. In order to complete the proof, it
suffices to show that we can choose some $\sigma\in\Gal(\tE/F)$ with
$\sigma|_E=\sigma_0$, and such that in the notation above, we have
$\pi_{\sigma_0}\circ\phi(\sigma)\circ i_{\sigma_0}\ne 0$, because we
can then choose $v$ to have $\Frob_v=\sigma$ by the Chebotarev density
theorem. 

To this end, choose any $\sigmat_0\in\Gal(\tE/F)$ with
$\sigmat_0|_E=\sigma_0$. If $\sigmat_0$ does not work, then we have
$\pi_{\sigma_0}\circ\phi(\sigmat_0)\circ i_{\sigma_0}=0$. In this
case, take $\sigma=\sigma_1\sigmat_0$ for some
$\sigma_1\in\Gal(\tE/E)$. Then
$\phi(\sigma)=\phi(\sigma_1\sigmat_0)=\phi(\sigma_1)+\sigma_1\phi(\sigmat_0)=\phi(\sigma_1)+\phi(\sigmat_0)$,
so $\pi_{\sigma_0}\circ\phi(\sigma)\circ
i_{\sigma_0}=\pi_{\sigma_0}\circ\phi(\sigma_1)\circ i_{\sigma_0}$.

Note that $\phi(\Gal(\tE/E))$ is a $\Gal(E/F)$-invariant subset of
$\ad^0\rhobar$, which is an irreducible $\Gal(E/F)$-module, since the
image of $\rhobar$ contains $\SL_2(\F_p)$. Thus the $\F$-span of
$\phi(\Gal(\tE/E))$ is all of $\ad^0\rhobar(1)$, from which it is
immediate that we can choose $\sigma_1$ so that
$\pi_{\sigma_0}\circ\phi(\sigma_1)\circ i_{\sigma_0}\ne 0$.
\end{proof}

We are now surprisingly close to proving the main theorem! Write
$h:=\#T-1-[F:\Q]+r$, and $R_\infty:=\Rloc\llbracket x_1,\dots,x_h\rrbracket $. For each
set $Q_N$ as above, choose a surjection $R_\infty\onto
R_{Q_N}^\square$. Let $\cJ_\infty:=\cJ\llbracket y_1,\dots,y_r\rrbracket $. Choose a 
surjection $\cJ_\infty\onto\cJ[\Delta_{Q_N}]$, given by writing
$Q_N=\{v_1,\dots, v_r\}$ and mapping $y_i$ to $(\gamma_i-1)$, where
$\gamma_i$ is a generator of $\Delta_{v_i}$. Choose a homomorphism
$\cJ_\infty\to R_\infty$ so that the composites $\cJ_\infty\to
R_\infty\onto R_{Q_N}^{\square}$ and
$\cJ_\infty\to\cJ[\Delta_{Q_N}]\to R_{Q_N}^\square$ agree, and write
$\ca_\infty:=(\ca,y_1,\dots,y_r)$. Then
$S_{Q_N}^\square/\ca_\infty=S_\emptyset$,
$R_{Q_N}^\square/\ca_\infty=R_\emptyset^\univ$. 

Write $\cb_N:=\ker(\cJ_\infty\to\cJ[\Delta_{Q_N}])$, so that
$S_{Q_N}^\square$ is finite free over $\cJ_\infty/\cb_N$. Since all
the elements of $Q_N$ are congruent to $1$ modulo $p^N$, we see that
$\cb_N\subseteq ((1+y_1)^{p^N}-1,\dots,(1+y_r)^{p^N}-1)$.

We can and do choose the same data for $R^{\operatorname{loc},'}$, in
such a way that the two sets of data are compatible modulo $\lambda$.

Now choose open ideals $\cc_N\lhd\cJ_\infty$ such that
\begin{itemize}
\item $\cc_N\cap\cO=(\lambda^N)$.
\item $\cc_N\supseteq\cb_N$.
\item $\cc_N\supseteq\cc_{N+1}$.
\item $\cap_N\cc_N=0$.
\end{itemize}
(For example, we could take
$\cc_N=((1+X_{v,i,j})^{p^N}-1,(1+y_i)^{p^N}-1,\lambda^N)$.) Note that
since $\cc_N\supseteq\cb_N$, $S_{Q_N}^\square/\cc_N$ is finite free
over $\cJ_\infty/\cc_N$. Also choose open ideals $\cd_N\lhd
R_\emptyset^\univ$ such that
\begin{itemize}
\item $\cd_N\subseteq\ker(R_\emptyset^\univ\to\End(S_\emptyset/\lambda^N))$.
\item $\cd_N\supseteq\cd_{N+1}$.
\item $\cap_N\cd_N=0$.
\end{itemize}
If $M\ge N$, write $S_{M,N}=S_{Q_M}^\square/\cc_N$, so that $S_{M,N}$ is finite free over
$\cJ_\infty/\cc_N$ of rank equal to the $\cO$-rank of $S_\emptyset$;
indeed $S_{M,N}/\ca_\infty\isoto S_\emptyset/\lambda^N$. Then we have a
commutative diagram \tikzset{every loop/.style={min distance=10mm,in=70,out=110,looseness=10}}
\[\begin{tikzcd}\cJ_\infty\arrow{r}&R_\infty\ar[two heads]{r}&
  R_\emptyset^\univ/\cd_N\\ & S_{M,N}\arrow[loop above]{} \ar[two heads]{r}&S_\emptyset/\cd_N\arrow[loop above]{}
\end{tikzcd}\] where
$S_{M,N}$, $S_\emptyset/\cd_N$ and $R_\emptyset^\univ/\cd_N$ all have
finite cardinality. Because of this finiteness, we see that there is
an infinite subsequence of pairs $(M_i,N_i)$ such that $M_{i+1}>M_i$,
$N_{i+1}>N_i$, and
the induced diagram \[\begin{tikzcd}\cJ_\infty\arrow{r}&R_\infty\ar[two heads]{r}&
  R_\emptyset^\univ/\cd_{N_i}\\ & S_{M_{i+1},N_{i+1}}/\cc_{N_i}\arrow[loop above]{} \ar[two heads]{r}&S_\emptyset/\cd_{N_i}\arrow[loop above]{}
\end{tikzcd}\]  
is isomorphic to the diagram for $(M_i,N_i)$.

Then we can take the projective limit over this subsequence, to
obtain a commutative diagram \[\begin{tikzcd}\cJ_\infty\arrow{r}&R_\infty\ar[two heads]{r}&
  R_\emptyset^\univ\\ & S_\infty\arrow[loop above]{} \ar[two heads]{r}&S_\emptyset\arrow[loop above]{}
\end{tikzcd}\] where
$S_\infty$ is finite free over $\cJ_\infty$. Furthermore, we can
simultaneously carry out the same construction in the $'$ world,
compatibly with this picture modulo $\lambda$.

This is the key picture, and the theorem will now follow from it by
purely commutative algebra arguments. We have (ultimately by the
calculations of the dimensions of the local deformation rings in
Theorems~\ref{thm:FL deformation rings} and~\ref{thm: types are constant on components
    etc}) $\dim R_\infty=\dim
R'_\infty=\dim \cJ_\infty=4\#T+r$, and since $S_\infty, S'_\infty$ are finite
free over the power series ring $\cJ_\infty$ (from Proposition~\ref{fact: two actions of Delta agree}), we have
$\depth_{\cJ_\infty}(S_\infty)=\depth_{\cJ_\infty}(S'_\infty)=4\#T+r$.
(This is the ``numerical coincidence'' on which the Taylor--Wiles
method depends; see~\cite{CalGer} for a further discussion of this
point, and of a more general ``numerical coincidence''.) Since
the action of $\cJ_\infty$ on $S_\infty$ factors through $R_\infty$,
we see that $\depth_{R_\infty}(S_\infty)\ge 4\#T+r$, and similarly
$\depth_{R'_\infty}(S'_\infty)\ge 4\#T+r$. Now, if $\cp\lhd R'_\infty$
is a minimal prime in the support of $S'_\infty$, then we see
that \[4\#T+r=\dim R'_\infty\ge\dim
R'_\infty/\cp\ge\depth_{R'_\infty}S'_\infty\ge 4\#T+r,\] so
equality holds throughout, and $\cp$ is a minimal prime of
$R'_\infty$. But $R'_\infty$ has a unique minimal prime, so in fact
$\Supp_{R'_\infty}(S'_\infty)=\Spec R'_\infty$.

By the same argument, we see that $\Supp_{R_\infty}(S_\infty)$ is a
union of irreducible components of $\Spec R_\infty$. We will show that
it is all of $\Spec R_\infty$ by reducing modulo $\lambda$ and
comparing with the situation for $S'_\infty$.

To this end, note that since  $\Supp_{R'_\infty}(S'_\infty)=\Spec
R'_\infty$, we certainly have
\[\Supp_{R'_\infty/\lambda}(S'_\infty/\lambda)=\Spec
R'_\infty/\lambda.\] By the compatibility between the two pictures,
this means that  \[\Supp_{R_\infty/\lambda}(S_\infty/\lambda)=\Spec
R_\infty/\lambda.\] Thus $\Supp_{R_\infty}(S_\infty)$ is a union of
irreducible components of $\Spec R_\infty$, which contains the
entirety of $\Spec R_\infty/\lambda$. Since (by Theorem~\ref{thm:Ihara avoidance deformations}) the irreducible components
of $\Spec R_\infty/\lambda$ are in bijection with the irreducible
components of $\Spec R_\infty$, this implies that
$\Supp_{R_\infty}(S_\infty)=\Spec R_\infty$. Then
$\Supp_{R_\infty/\ca_\infty}(S_\infty/\ca_\infty)=R_\infty/\ca_\infty$,
i.e. $\Supp_{R_\emptyset^\univ}S_\emptyset=R_\emptyset^\univ$, which
is what we wanted to prove.

\section{Relaxing the hypotheses}\label{subsec:relaxing hypotheses}The hypotheses in our main theorem
are not optimal. We will now briefly indicate the ``easy''
relaxations of the assumptions that could be made, and discuss the
generalisations that are possible with (a lot) more work. 

Firstly, it is possible to relax the assumption that $p\ge 5$, and
that $\Im\rhobar\supseteq\SL_2(\Fp)$. These assumptions cannot be
completely removed, but they can be considerably relaxed. The case
$p=2$ is harder in several ways, 
but important theorems
have been proved in this case, for example the results
of Kisin~\cite{MR2551765} which completed the proof of Serre's conjecture.

On the other hand, the case $p=3$ presents no real difficulties. The main place that we
assumed that $p>3$ was in the proof that the finite groups~$G_i$ in
Section~\ref{subsec:
  integral auto forms} have order prime to~$p$; this argument could also break
down for cases when $p>3$ if we allowed $p$ to ramify in $F$, which in
general we would like to do. Fortunately, there is a simple solution
to this problem, which is to introduce an auxiliary prime~ $v$ to the
level. 
This prime is chosen in such a way that all
deformations of $\rhobar|_{G_{F_v}}$ are automatically unramified, so
none of the global Galois deformation rings that we work with are
changed when we relax the conditions at $v$. The existence of an
appropriate $v$ follows from the Chebotarev density theorem and some
elementary group theory; see Lemma 4.11 of~\cite{MR1605752} and the
discussion immediately preceding it.

We now consider the possibility of relaxing the assumption that
$\Im\rhobar\supseteq\SL_2(\Fp)$. We should certainly assume that
$\rhobar$ is absolutely irreducible, because otherwise many of our
constructions don't even make sense; we always had to assume this in
constructing universal deformation rings, in constructing the
universal modular deformation, and so on. (Similar theorems have been
proved in the case that $\rhobar$ is reducible, in particular by
Skinner--Wiles
\cite{MR1793414},
but the arguments are considerably more involved, and at present
involve a number of serious additional hypotheses, in particular
ordinarity --- although see Pan's \cite{pan2021fontainemazur} for a theorem
without an ordinarity hypothesis.) Examining the arguments made above, we see that the main
use of the assumption that $\Im\rhobar\supseteq\SL_2(\Fp)$ is
in the proof of Proposition~\ref{prop: Chebotarev argument for TW
  primes}. Looking more closely at the proof, the key assumption is
really that $\rhobar|_{G_{F(\zeta_p)}}$ is absolutely irreducible; this is known
as the ``Taylor--Wiles assumption''. (Note that by elementary group
theory, this is equivalent to the absolute irreducibility of
$\rhobar|_{G_K}$, where $K/F$ is the unique quadratic subextension of
$F(\zeta_p)/F$; in particular, over $\Q$ the condition is equivalent
to the absolute irreducibility of
$\rhobar|_{G_{\Q(\sqrt{(-1)^{(p-1)/2}p})}}$, which is how the
condition is stated in the original papers.)

Unfortunately this condition isn't quite enough in complete
generality, but it comes very close; the only exception is certain
cases when $p=5$, $F$ contains $\Q(\sqrt{5})$, and the projective
image of $\rhobar$ is $\PGL_2(\F_5)$. See~\cite[(3.2.3)]{kis04} for
the definitive statement (and see the work of Khare--Thorne
\cite{MR3648503} for some improvements in this exceptional case). If
$\rhobar$ is absolutely irreducible, but $\rhobar|_{G_{F(\zeta_p)}}$
is (absolutely) reducible, it is sometimes possible to prove
modularity lifting theorems, but considerably more work is needed (and
there is no general approach in higher dimension); see~\cite{MR1928993} in the ordinary
case, which uses similar arguments to those of~\cite{MR1793414}, and
also \cite{MR3451399, pan2021fontainemazur}.

The other conditions that we could hope to relax are the assumptions
on $\rho|_{G_{F_v}}, \rho_0|_{G_{F_v}}$ at places $v|p$. We've hardly
discussed where some of these assumptions come from, as we swept most
issues with $p$-adic Hodge theory under the carpet. There are
essentially two problems here. One is that we have assumed that $p$ is
unramified in $F$, that the Galois representations are crystalline,
and that the gaps between the Hodge--Tate weights are ``small''; this
is the Fontaine--Laffaille condition. There is also the assumption
that $\rho, \rho_0$ have the same Hodge--Tate weights. Both conditions
can be considerably (although by no means completely) relaxed (of
course subject to the necessary condition that $\rho$ is geometric). As
already alluded to above, very general results are available in the
ordinary case (even in arbitrary dimension), in particular those of
Geraghty~\cite{ger}. In the case that $F_v=\Qp$ there are again very
general results, using the $p$-adic local Langlands correspondence
for~$\GL_2(\Qp)$ (see in particular~\cite{unpublished,MR2505297,
  pan2021fontainemazur}). However, beyond this case, the situation is
considerably murkier, and at present there are no generally applicable results.

\subsection{Further generalisations}
Other than the results discussed in the
previous subsection, there are a number of obvious generalizations
that one could hope to prove. One obvious step, already alluded to
above, is to replace $2$-dimensional representations with
$n$-dimensional representations; we could also hope to allow~$F$ to be a more general number
field. At present it seems to be necessary to assume that~$F$ is a CM
field, as otherwise we do not know how to attach Galois
representations to automorphic representations; but if~$F$ is CM, then
automorphy lifting theorems analogous to our main theorem are now
known (for arbitrary~$n$), and we refer
to~\cite{calegari2021reciprocity} for both the history of such results
and the state of the art.



Another natural condition to relax would be the condition that the
Hodge--Tate weights are distinct; for example, one could ask that they
all be equal, and hope to prove Artin's conjecture, or that some are
equal, to prove modularity results for abelian varieties. The
general situation where some Hodge--Tate weight occurs with multiplicity
greater than 2 seems to be completely out of reach (because there is
no known way to relate the automorphic representations expected to correspond to
such Galois representations to the automorphic representations which
contribute to the cohomology of Shimura varieties, which is the only
technique we have for constructing the maps $R\to\T$), but there has been
considerable progress for small dimensional cases, for which we again
refer the reader to~\cite{calegari2021reciprocity}.  

Finally, we would of course like to be able to dispose of
the hypothesis that $\rhobar$ is modular (that is, to dispose of
$\rho_0$). This is the problem of Serre's conjecture and its
generalisations, and has only been settled in the case that $F=\Q$ and
$n=2$. The proof in that case (by Khare--Wintenberger and Kisin, see
\cite{MR2551763, MR2551764, MR2551765}) makes essential use of
modularity lifting theorems. The proof inductively reduces to the case
that $p\le 5$ and $\rhobar$ has very little ramification, when direct
arguments using discriminant bounds can be made. The more general
modularity lifting theorems mentioned above make it plausible that the
inductive steps could be generalised, but the base case of the
induction seems specific to the case of $\GL_2/\Q$, and proving the
modularity of~$\rhobar$ in greater generality is one of the biggest
open problems in the field.

\bibliographystyle{amsalpha}
\bibliography{tobylecturesbib.bib}
\end{document}